\documentclass[12pt]{article}

\usepackage{amsmath, amsthm,amssymb,fullpage,pifont,enumitem,tikz,adjustbox}
\usepackage{xcolor,hyperref,cleveref}
\usetikzlibrary{arrows.meta,positioning,bending,decorations.pathmorphing}

\newtheorem{thm}{Theorem}
\newtheorem{lemma}[thm]{Lemma}
\newtheorem{definition}[thm]{Definition}

\renewcommand{\S}{\mathcal{S}}
\newcommand{\K}{\mathcal{K}}
\newcommand{\ep}{\varepsilon}
\newcommand{\C}{\mathcal{C}}
\newcommand{\D}{\mathcal{D}}
\newcommand{\N}{\mathbb{N}}

\newcommand{\sbeq}{\subseteq}
\newcommand{\forw}[1]{(#1)^\rightarrow}
\newcommand{\back}[1]{(#1)^\leftarrow}

\newcommand{\arc}[2]{(#1)^{#2}}
\newcommand{\arcset}{\{\rightarrow,\leftarrow\}}

\newcommand{\Vres}{V_{\mathrm{res}}}
\newcommand{\Cres}{\mathcal{C}_{\mathrm{res}}}

\newcommand{\floor}[1]{\left\lfloor#1\right\rfloor}
\newcommand{\ceil}[1]{\left\lceil#1\right\rceil}

\title{Universality for rainbow oriented cycles in perturbed digraphs}
\author{Robert A.~Krueger\footnote{Department of Mathematical Sciences, Carnegie Mellon University, Pittsburgh, PA, USA. Research supported by NSF Award DMS-2402204. Email: \texttt{rkrueger@andrew.cmu.edu.}} \and David Staudinger\footnote{Department of Mathematical Sciences, Carnegie Mellon University, Pittsburgh, PA, USA. Email: \texttt{dstaudin@andrew.cmu.edu.}}}
\date{}

\begin{document}

\maketitle

\begin{abstract}
A randomly perturbed digraph is an $n$-vertex directed graph with all out- and in-degrees linear in $n$, to which a linear number (depending on the degree) of random edges have been randomly added. We show that randomly perturbed digraphs whose edges have been colored uniformly with $n$ colors have a rainbow copy of every orientation of every possible length cycle, simultaneously, with high probability. This is a common generalization of work of Araujo, Balogh, Krueger, Piga, and Treglown in the uncolored setting and Katsamaktsis, Letzter, and Sgueglia for consistently oriented spanning cycles. Our proof uses Montgomery's distributive absorption method.
\end{abstract}

\section{Introduction}\label{sec:intro}

Determining whether a given graph has a spanning cycle, a cycle which contains all the vertices of a graph (also called a Hamiltonian cycle), is one of Karp's original NP-complete problems~\cite{Karp}. Given this, it is natural to ask for sufficient conditions for a graph to contain a spanning cycle, and there are several classical theorems to this effect. A foundational example from extremal graph theory is Dirac's Theorem~\cite{Dirac}, which states that every $n$-vertex graph $G$ with \emph{minimum degree} $\delta(G) \geq n/2 > 1$ contains a spanning cycle. This is best possible, since $K_{\floor{(n-1)/2}, \ceil{(n+1)/2}}$ and the union of $K_{\floor{(n+1)/2}}$ and $K_{\ceil{(n+1)/2}}$ which share only one vertex are both $n$-vertex graphs with minimum degree $\floor{(n-1)/2}$ and no spanning cycle, the reasons being that the former has too large an independent set, and the latter is not $2$-connected.

Alternatively, if the edges of the graph are randomly distributed, then much fewer edges are needed. Let $G(n,p)$ be the \emph{binomial (Erd\H{o}s--R\'enyi) random graph} on $n$ vertices, where each possible edge is included independently with probability $p$. P\'osa~\cite{Posa} showed that for $p = C\log n/n$ with $C$ a sufficiently large constant, $G(n,p)$ contains a spanning cycle \emph{with high probability (w.h.p.)}, that is, with probability tending to $1$ as $n$ tends to infinity. As made precise by a `hitting time' theorem~\cite{AKSz,Bollobas}, the main obstacle preventing $G(n,p)$ from having a spanning cycle is a vertex of degree at most $1$.

Bohman, Frieze, and Martin~\cite{BFM} introduced an interpolation between Dirac's Theorem and P\'osa's result: a \emph{randomly perturbed graph} is a graph on top of which some random edges have been added. They showed that for every $\alpha > 0$, there exists $C > 0$ such that if $G_0$ is an $n$-vertex graph with $\delta(G_0) \geq \alpha n$, then $G_0 \cup G(n,C/n)$ contains a spanning cycle w.h.p., where we take the vertex set of the `deterministic' graph $G_0$ and the random graph $G(n,C/n)$ to be the same. An interesting feature of these randomly perturbed graphs is that they require less random and less deterministic edges than the purely random or purely deterministic counterparts. Informally, this is because the deterministic edges avoid problematic vertices of low degree, while the random edges avoid large independent sets and low connectivity. Note that, up to the dependence of $C$ on $\alpha$, this theorem is tight: if $G_0 = K_{\alpha n,(1-\alpha)n}$, then a linear number of random edges are needed to contain a spanning cycle w.h.p. Since their introduction, randomly perturbed graphs have been the subject of much research in probabilistic combinatorics, often generalizing from spanning cycles to other structures. Even this initial result of Bohman, Frieze, and Martin has come under close analysis recently, with a more precise understanding of how many random edges are needed if $\alpha = o(1)$~\cite{HKMMMP} or $\alpha = 1/2 - o(1)$~\cite{EDVR}.

\subsection{Spanning cycles in directed graphs}

For a directed graph $D$, we denote by $\delta^0(D)$ the \emph{minimum semi-degree} of $D$, which is the minimum of the minimum in-degree of $D$ and the minimum out-degree of $D$. A directed analogue of Dirac's Theorem is given by a corollary of a theorem of Ghouila-Houri~\cite{GH}: every $n$-vertex digraph $D$ with $\delta^0(D) \geq n/2$ contains a \emph{consistently oriented} spanning cycle, that is, a cycle $v_1, v_2, \dots, v_n$ where the edges are directed as $\forw{v_{i} v_{i+1}}$ (from $v_i$ to $v_{i+1}$) for all $i \in [n]$ (index addition is considered modulo $n$). Similar to Dirac's Theorem, this is best possible. Much later and with modern techniques, DeBiasio, K\"uhn, Molla, Osthus, and Taylor~\cite{DKMOT} showed that this semi-degree is enough to guarantee every orientation of a spanning cycle (at least for $n$ sufficiently large), except possibly the \emph{alternating orientation}, where $n$ is even and the edges alternate between backwards and forwards around the cycle. DeBiasio and Molla~\cite{DM} showed that $\delta^0(D) \geq n/2 + 1$ forces the alternating spanning cycle, and this semi-degree condition is best possible. It is worth noting that the classical rotation/extension technique used in for Dirac's Theorem and P\'osa's proof for spanning cycles in $G(n,p)$ can be made to work for consistently oriented cycles, but not for arbitrary orientations.

Let $D(n,p)$ denote the \emph{binomial random digraph} on $n$ vertices, where each of the $n(n-1)$ possible edges are included independently with probability $p$. A clever coupling of $D(n,p)$ with $G(n,p)$ due to McDiarmid~\cite{McDiarmid} gives good bounds on the appearance of any particular orientation of a spanning cycle in $D(n,p)$. The precise threshold was determined by Montgomery~\cite{MontCycles}, who additionally showed a \emph{universality} result: if the edges are added randomly and sequentially, by the time the random digraph contains the consistently oriented spanning cycle, it also contains every other orientation of a spanning cycle, w.h.p.

In their original paper on randomly perturbed graphs, Bohman, Frieze, and Martin~\cite{BFM} also showed that randomly perturbed digraphs contain the consistently oriented spanning cycle w.h.p. The following theorem extends this result to all orientations, and moreover achieves universality, finding all orientations simultaneously.

\begin{thm}[Araujo, Balogh, Krueger, Piga, Treglown~\cite{ABKPT}]\label{thm:ABKPT}
For every $\alpha > 0$, there exists $C > 0$ such that if $D_0$ is an $n$-vertex digraph with minimum semi-degree $\delta^0(D_0) \geq \alpha n$, then $D_0 \cup D(n,C/n)$ contains every orientation of every cycle of length between $2$ and $n$, simultaneously, with high probability.
\end{thm}

\subsection{Rainbow spanning cycles}

An edge-colored (di)graph is called \emph{rainbow} if all the colors appearing on its edges are distinct. (Here we will always be coloring the edges of the graphs.) The study of rainbow subgraphs goes all the way back to Euler's study of transversals in Latin squares, which are equivalent to rainbow perfect matchings in properly $n$-colored $K_{n,n}$ (and equivalently, $1$-factors in a `properly' $n$-colored $n$-vertex complete directed graph with loops). When any `reasonable' coloring (such as proper edge-colorings or colorings with a bounded number of edges of each color) contains a desired rainbow structure, this is known as an `anti-Ramsey' property. Pertaining to spanning cycles, Andersen conjectured~\cite{Andersen} that every proper edge-coloring of $K_n$ contains a rainbow path of length at least $n-2$, which would be best possible in the length of the path. See~\cite{Pokrovskiy} for a survey on Andersen's and related conjectures; in particular, see~\cite{BM} for the best result on this conjecture to date, and see~\cite{BPS} for a directed version.

The study of rainbow subgraphs also appears in the geometric and topological combinatorics literature as `colorful' or `transversal' versions of various classical theorems, like Carath\'eodory's Theorem and Helly's Theorem~\cite{Barany}. For example, Joos and Kim~\cite{JoosKim} recently proved a transversal version of Dirac's Theorem:  if $G_1, \dots, G_n$ are graphs on the same set of $n\geq 3$ vertices with $\delta(G_i) \geq n/2$ for all $i \in [n]$, then there exists a spanning cycle using exactly one edge from each $G_i$. See~\cite{SWW} for a survey of such transversal results in graphs, digraphs, and hypergraphs.

Often adversarial colorings can be difficult to deal with, so one option for relaxing such problems is to consider a random coloring. As it relates to Andersen's conjecture, Gould, Kelly, K\"uhn, and Osthus~\cite{GKKO} showed that a typical proper coloring of $K_n$ has a spanning rainbow path. The probability distribution of a random proper coloring of $K_n$ is quite difficult to work with, and so a natural further relaxation is a completely random coloring, where each edge receives a color independently from a fixed distribution. Random colorings are also natural from other perspectives; for example, a randomly $n$-colored $K_{n,n}$ is equivalent to a random $3$-uniform $3$-partite hypergraph where each part has $n$ vertices, with a rainbow matching corresponding to a hypergraph matching.

Cooper and Frieze~\cite{CF} showed that if $G(n,p)$ with $p = C\log n/n$ is randomly colored uniformly independently from a set of $Cn$ colors, then there exists a rainbow spanning cycle w.h.p.\ if $C$ is a sufficiently large universal constant. Bal and Frieze~\cite{BF} and Ferber~\cite{Ferber} improved the number of colors to the optimal $n$, while Frieze and Loh~\cite{FL} (sharpened further by Ferber and Krivelevich~\cite{FK}) obtained the same result with $p = (1+o(1))\log n/n$ and $(1+o(1))n$ colors, which are both asymptotically optimal. As in the uncolored version, a coupling trick due to McDiarmid allows one to transfer these results from rainbow spanning cycles in randomly colored $G(n,p)$ to arbitrary orientations of rainbow spanning cycles in randomly colored $D(n,p)$, although Montgomery's universality~\cite{MontCycles} does not immediately transfer to the colored setting.

As for rainbow spanning cycles in randomly perturbed graphs, Anastos and Frieze~\cite{AF} showed that for every $\alpha > 0$, there exists $C_1$ and $C_2$ such that if $G_0$ is an $n$-vertex graph with $\delta(G_0) \geq \alpha n$, then when $G_0 \cup G(n,C_1/n)$ is randomly colored with $C_2n$ colors, there exists a rainbow spanning cycle with high probability. Aigner-Horev and Hefetz~\cite{AHH} improved the number of colors to $(1+o(1))n$, while Katsamaktsis, Letzter, and Sgueglia~\cite{KLS} improved the number of colors to exactly $n$. In fact, they showed the following directed version for consistently oriented cycles, and use McDiarmid's coupling to pull the result back to the undirected setting.

\begin{thm}[Katsamaktsis, Letzter, Sgueglia~\cite{KLS}]\label{thm:KLS}
For every $\alpha > 0$, there exists $C > 0$ such that if $D_0$ is an $n$-vertex digraph with minimum semi-degree $\delta^0(D_0) \geq \alpha n$ and $D_0 \cup D(n,C/n)$ is uniformly edge-colored from $[n]$, then $D_0 \cup D(n,C/n)$ contains a consistently oriented rainbow spanning cycle, with high probability.
\end{thm}

Our main result is the common generalization of \Cref{thm:ABKPT} and \Cref{thm:KLS}.

\begin{thm}\label{thm:main}
For every $\alpha > 0$, there exists $C > 0$ such that if $D_0$ is an $n$-vertex digraph with minimum semi-degree $\delta^0(D_0) \geq \alpha n$ and $D_0 \cup D(n,C/n)$ is uniformly edge-colored from $[n]$, then $D_0 \cup D(n,C/n)$ contains a rainbow copy of every orientation of every cycle of length between $2$ and $n$, simultaneously, with high probability.
\end{thm}

The proofs of \Cref{thm:ABKPT}, \Cref{thm:KLS}, and \Cref{thm:main} are all based on the absorbing method, which is a class of techniques whose goal is to turn an almost-spanning structure into a spanning structure. In this setting, finding an almost-spanning structure is fairly easy (see \Cref{lem:longpath}), so the heart of the matter is in setting up an absorber. The absorbing method was first laid out by R\"odl, Ruci\'nski, and Szemer\'edi~\cite{RRSz}, and has been developed into a host of very powerful techniques, used to solve many famous conjectures, such as the existence of designs~\cite{Keevash,GKLO,DP}. We use a type of absorption called distributive absorption (sometimes called template absorption), invented by Montgomery~\cite{Mont}. While the proofs of both \Cref{thm:ABKPT} in~\cite{ABKPT} and \Cref{thm:KLS} in~\cite{KLS} also use distributive absorption, their proofs differ from ours in interesting ways. Our use of distributive absorption is similar to that in~\cite{ABKPT}, while their `absorbing gadgets' were relatively simple compared to ours and make crucial use of a coupling which is unavailable to us (see the Remark below). Distributive absorption is used quite differently in~\cite{KLS} compared to us, while their gadget inspires the design of one of our gadgets. Of independent interest may be our more general framework for finding our gadgets in randomly perturbed digraphs, which avoids the poor parametric dependencies caused by the use of Szemer\'edi's Regularity Lemma in~\cite{KLS}. We sketch our proof in \Cref{sec:sketch}.

\paragraph{Remark.}\label{Remark1}
A natural approach to achieving a universality result is to show that the probability that any one individual object is not present is incredibly small. Indeed, this is how the proof of \Cref{thm:ABKPT} runs: they show that $D_0 \cup D(n,C/n)$ contains a given oriented cycle with probability at least $1-e^{-n}$, which is small enough that the union bound over all at most $n2^n$ oriented cycles goes through. They use McDiarmid's coupling to treat the random edges of $D_0 \cup D(n,C/n)$ as `bidirected,' making it significantly easier to find a given oriented cycle. It is suggested in~\cite{KLS} that their method together with ideas from~\cite{ABKPT} could prove a rainbow version of such a result, that is, they suggest to prove that with probability at least $1-e^{-n}$, a randomly colored $D_0 \cup D(n,C/n)$ contains a rainbow copy of a given oriented cycle. Unfortunately, this cannot be true: the probability that $D_0 \cup D(n,C/n)$ misses a given color is approximately $(1-1/n)^{\alpha n^2} \approx \exp(-\alpha n)$, in which case we cannot find a rainbow spanning cycle of any orientation. This prevents us from using `bidirected' edges via the McDiarmid coupling strategy from~\cite{ABKPT} and forces us to obtain universality head-on.

\paragraph{Organization.} We give a sketch of the proof of \Cref{thm:main} in \Cref{sec:sketch}. In \Cref{sec:prelim}, we state and prove some preliminary lemmas about concentration, randomly perturbed graphs, and supersaturation which are used in the construction of the absorber. \Cref{sec:gadgets}, \Cref{sec:local_abs}, and \Cref{sec:global_abs} concern the construction of the absorber; each subsequent section uses the constructions of the previous section as a black box to construct a more versatile absorber. In \Cref{sec:mainproof}, we apply the absorber to prove \Cref{thm:main}.

\subsection{Notation}\label{sec:intro:notation}

A \emph{digraph} $D$ consists of a vertex set $V(D)$ and a set $E(D) \subseteq \{ (u,v) : u,v \in V(D)\}$ of ordered pairs of vertices called (directed) edges. Our digraphs will not have multiple edges or loops (edges of the form $(v,v)$), but it is possible to have both $(u,v)$ and $(v,u)$ as distinct edges. We denote the edge from $u$ to $v$ by $\forw{uv}$ or by $\back{vu}$; for $\star \in \arcset$, we denote by $(uv)^\star$ the appropriate forward or backward edge, and we denote by $-\star$ the opposite orientation of $\star$. For $\star \in \arcset$ and $U,V \subseteq V(D)$, we let $E^\star(U,V) = \{(uv)^\star \in E(D) : u \in U, v \in V\}$.

The \emph{out-neighborhood} (resp.\ \emph{in-neighborhood}) of a vertex $v \in V(D)$, denoted $N^{\rightarrow}(v)$ (resp.\ $N^{\leftarrow}(v)$), is $\{u \in V(D) : \forw{vu} \in E(D)\}$ (resp.\ $\back{vu}$). The \emph{out-degree} (resp.\ \emph{in-degree}) of $v$, denoted $d^{\rightarrow}(v)$ (resp.\ $d^{\leftarrow}(v)$), is $|N^{\rightarrow}(v)|$ (resp.\ $|N^{\leftarrow}(v)|$). The minimum out-degree (resp.\ in-degree) of $D$, denoted $\delta^{\rightarrow}(D)$ (resp.\ $\delta^{\leftarrow}(D)$), is the minimum of $d^{\rightarrow}(v)$ (resp.\ $d^{\leftarrow}(v)$) over all $v \in V(D)$. The \emph{minimum semi-degree} of $D$, denoted $\delta^0(D)$, is the minimum of $\delta^{\rightarrow}(D)$ and $\delta^{\leftarrow}(D)$.

A \emph{(directed) path} $P$ is a digraph with a specified ordering of its vertices $v_1, \dots, v_{k+1}$ such that the only edges are $\arc{v_i v_{i+1}}{\star_i}$ for some $\star_i \in \arcset$, for every $i \in [k]$. We say that $P$ has \emph{length} $k$, \emph{orientation} $\star = (\star_1,\dots,\star_k) \in \arcset^k$, \emph{startpoint} $v_1$, and \emph{endpoint} $v_{k+1}$. We call $v_2, \dots, v_k$ the \emph{interior} vertices of $P$. We say that $P$ is \emph{consistently oriented} if $\star_i = \rightarrow$ for every $i$, and we say that $P$ is \emph{alternating} if $\star_i \neq \star_{i+1}$ for every $i$. 

An edge-colored digraph $D$ is called \emph{rainbow} if no pair of edges of $D$ have the same color, including edges between the same vertices but in opposite directions. For an edge-colored digraph $D$, let $\mathcal{C}(D)$ denote the set of colors on the edges of $D$. For $v \in V(D)$ and $\star \in \arcset$, we let $\C^\star(v)$ denote the set of colors on edges of the form $(vu)^\star$ for $u \in V(D)$. For a set of colors $\C$, we let $E_{\C}(D)$ denote the set of edges of $D$ with a color in $\C$. Additionally, for $\star \in \arcset$, and $v \in V(D)$, we let $N^\star_\C(v) = \{u \in V(D): (vu)^\star \in E_\C(D)\}$ and $d^\star_\C(v) = |N^\star_\C(v)|$. Whenever $\C$ consists of a just a single color $c$, we write $d^\star_c(v)$ instead of $d^\star_{\{c\}}(v)$, etc. 

By $D(n,p)$ we denote the $n$-vertex random digraph where each edge is present independently with probability $p$. Whenever we discuss randomly perturbed digraphs $D_0 \cup D(n,p)$, we assume that $D_0$ is an $n$-vertex graph and that $D_0$ and $D(n,p)$ have the same vertex set. When we randomly color $D_0 \cup D(n,p)$, we assume that each edge receives a color uniformly from a set of $n$ colors, and that the edges receive their colors independently of each other and of the randomness in $D(n,p)$. We say that an event happens `with high probability' (abbreviated w.h.p.) if the probability of that event goes to $1$ as $n$ goes to infinity.

Throughout our various lemmas, many parameters make appearances; we usually use the same Greek letter to denote parameters with the same function in different lemmas. Here is a glossary of parameters and their associations:
\begin{itemize}
\item $\alpha$: minimum semi-degree
\item $\nu$: forbidden set of vertices or colors
\item $\mu$, $\beta$: size of the global absorber
\item $\zeta$, $\eta$: size of leftover vertex/color set
\item $\ep$: pseudorandomness condition
\item $C$: random edge probability
\end{itemize}
Parameters with the same letters appearing in the different lemmas will have slightly different values; for example, as the argument progresses, more vertices are forbidden, and so $\nu$ will slightly increase. The parameters are globally ordered as follows:
\[ 1/n \ll 1/C \ll \ep \ll \eta, \zeta \ll \mu,\beta \ll \nu \ll \alpha ,\]
where $\beta \ll \gamma$ means that for every choice of $\gamma$, there exists a sufficiently small $\beta$ which makes the argument work. Usually the dependence among the parameters is polynomial, and the exact dependence can be inferred from the proofs. We ignore floors and ceilings wherever they are inconsequential, assuming an expression like $\alpha n$ is always an integer.

\section{Proof sketch}\label{sec:sketch}

The absorption method was first formally introduced by R\"odl, Ruci\'nski, and Szemer\'edi~\cite{RRSz} and has since become a fundamental technique in extremal and probabilistic combinatorics with many impressive applications. The purpose of the technique is to turn an almost spanning structure into a spanning structure. For our rainbow spanning cycles, the structure is spanning in two senses: it must contain all vertices, and it must contain all colors. As such, we build two separate absorbers, one to absorb vertices and one to absorb colors, but they share building blocks.

The original absorption method of~\cite{RRSz} was used to prove an analogue of Dirac's Theorem for $3$-uniform hypergraphs, and proceeded as follows. For each vertex $v$ in an $n$-vertex (hyper)graph, we find \emph{many} $O(1)$-size `gadgets' $A$ such that both $A$ and $A \cup \{v\}$ have spanning paths with the same start- and end-points. By randomly selecting $\Omega(n)$ gadgets across all $v$ (and thinning them to be vertex disjoint and connecting them) to create an `absorber' $\mathcal{A}$, we have that \emph{every} vertex $v$ has $\Omega(n)$ gadgets inside $\mathcal{A}$. This means that we can greedily absorb any $o(n)$ vertices into a path spanning $\mathcal{A}$. Given a nearly spanning cycle which incorporates $\mathcal{A}$, the leftover vertices not on the cycle can be absorbed into $\mathcal{A}$, yielding a spanning cycle.

In order for randomly selecting the gadgets to form the absorber to work, we really do need \emph{many} gadgets (quantified as $\Omega(n^{|A|})$ many) for every vertex, which randomly perturbed graphs usually do not have. Instead, we may substitute this `many gadgets' requirement with a weaker \emph{robust} local absorption property: for every vertex, there are $\Omega(n)$ \emph{disjoint} gadgets, or equivalently, for every vertex $v$ and $o(n)$-size set of forbidden vertices, there is a gadget for $v$ avoiding those forbidden vertices. With this property, we can construct an absorber for any $o(n)$-size set of vertices by greedily finding pairwise disjoint gadgets for each vertex (and connecting them). We build such an absorber on a special `reservoir' of $o(n)$ vertices; given a nearly spanning cycle on the graph excluding the reservoir and associated absorber, we incorporate the leftovers into cycle using the reservoir, and absorb the remaining vertices of the reservoir using the absorber.

This strategy breaks down when we cannot absorb just one vertex at a time, say, in the setting of \Cref{thm:KLS}: we need to be concerned with what colors are used depending on how we `activate' a gadget $A$ for $v$, that is, whether we use the path in $A$ or the path in $A \cup \{v\}$. Since these paths have different lengths, the (natural) approach taken by Katsamaktsis, Letzter, and Sgueglia~\cite{KLS} was to find gadgets $A$ for vertex-color pairs $(v,c)$, where $A$ and $A \cup \{v\}$ have spanning paths with the same start- and end-points and using the same set of colors, aside from $c$, which is used only on the path from $A \cup \{v\}$. There is simply not enough space to find pairwise \emph{disjoint} gadgets for quadratically many vertex-color pairs. To get around this in~\cite{KLS}, they build an absorber for a special linear-size set $\mathcal{P}$ of vertex-color pairs which is `robustly matchable,' meaning that for any set of remaining vertices and colors, there exists a perfect matching between the remaining vertices and colors such that every matched pair is in $\mathcal{P}$. This $\mathcal{P}$ was sourced from a sparse robustly matchable bipartite graph due to Montgomery~\cite{Mont} (see \Cref{lem:Mont}).

A different complication arises for arbitrary orientations of directed cycles, as in~\cite{ABKPT}: depending on how an gadget is activated, we may vary the length of the path found, and thus alter the pattern of the orientation for the subsequent gadgets. A solution to this problem is to construct gadgets which always produce the same length of path, no matter how they are activated. This means that we cannot have a gadget which absorbs a vertex or not; rather, in~\cite{ABKPT}, they used gadgets which absorbed \emph{exactly} one vertex out of a $O(1)$-size set. (To control the number of different $O(1)$-size sets they needed to absorb, they used the same sparse robustly matchable bipartite graph of Montgomery as a scaffolding, although in a different way from~\cite{KLS}.) They were able to find such gadgets fairly easy by exploiting the McDiarmid coupling to a model where the random edges were `bidirected,' so that the pattern of the orientation did not matter for them; still the main issue was planning for the directions of the deterministic edges. Since, in our goal of universality (see the remark at the end of \Cref{Remark1}), we are unable to exploit this coupling to get bidirected edges, the desired gadgets are much more difficult to construct.

In order to deal with the two previous issues, we developed the following definition of an $\mathcal{S}$-absorber, which has the key property that no matter how it is activated, it always produces the same length of path. For a set $S$ of vertices and/or colors, we let $V(S)$ denote the set of vertices appearing in $S$ and $\mathcal{C}(S)$ denote the set of colors appearing in $S$.

\begin{definition}[$\mathcal{S}$-absorber]\label{def:absorber}
Let $\mathcal{S} = \{S_1, \dots, S_m\}$ be a collection of sets $S_i$ such that each $S_i$ contains vertices and/or colors, $|V(S_1)| = \cdots = |V(S_m)|$, and $|\C(S_1)| = \cdots = |\C(S_m)|$. Given $k \in \mathbb{Z}^+$ and $\star \in \arcset^k$, an \emph{$(\mathcal{S},\star)$-absorber} is a triple $(A,s,t)$, where $A$ is a colored digraph, $s \in V(A)$ is a specified \emph{start} vertex, and $t \in V(A)$ is a specified \emph{end} vertex, such that there exists rainbow directed paths $P_1, \dots, P_m$ in $A$ satisfying the following:
\begin{itemize}
\item each $P_i$ has length $k$, orientation $\star$, starting at $s$, and ending at $t$,
\item for every $i \in [m]$, $V(P_i) = (V(A) \setminus V(\bigcup \mathcal{S})) \cup V(S_i)$,
\item for every $i \in [m]$, $\mathcal{C}(P_i) = (\mathcal{C}(A) \setminus \mathcal{C}(\bigcup \mathcal{S})) \cup \mathcal{C}(S_i)$.
\end{itemize}
We call $P_i$ the \emph{$S_i$-absorbing path}. We call $V(A) \setminus V(\bigcup \mathcal{S})$ the \emph{internal vertices} of $A$ and $\mathcal{C}(A) \setminus \mathcal{C}(\bigcup\mathcal{S})$ the \emph{internal colors} of $A$. Often, when $\star$ is understood, we just call $A$ (together with $s$ and $t$) an $\mathcal{S}$-absorber. If $\S$ consists of singletons, that is $\S=\{\{x_1\},\ldots, \{x_k\}\}$, then we omit the additional braces and write $\S=\{x_1,\ldots, x_k\}$. We say that an absorber $A$ is \emph{activated} along its $S_i$-absorbing path to indicate when we traverse from $s$ to $t$ along path $P_i$. 
\end{definition}

Since the absorbing paths of an $\mathcal{S}$-absorber all have the same length, no matter how the $\mathcal{S}$-absorber is activated (that is, which $S_i$ we absorb), the edges which come after the $\mathcal{S}$-absorber will always play the same role in the orientation of the spanning cycle. Thus only the edges within an $\mathcal{S}$-absorber may play multiple roles in the orientation.

We classify the $\mathcal{S}$-absorbers we construct into $3$ types, each of which is a building block for the next type: \emph{gadgets}, \emph{local absorbers}, and \emph{global absorbers}.
\begin{itemize}
\item In \Cref{sec:gadgets}, we construct gadgets, which allow us to absorb exactly one of a pair of vertices or colors. The design of our `color' gadgets is inspired by~\cite{KLS}, who in turn based their gadgets on~\cite{GKKO} in the undirected setting, while one of the designs of our `vertex' gadgets is inspired by~\cite{ABKPT}.
\item Our local absorbers are $\mathcal{S}$-absorbers with $|\mathcal{S}| = 40$, where each $S_i$ has only a few vertices or colors; these are constructed in \Cref{sec:local_abs}. To construct these for an arbitrary collection of vertices and colors, we first find a `special' $\mathcal{S}$-absorber, where we do not have much choice as to the vertices and colors in $\mathcal{S}$, and then we use our gadgets to swap out\footnote{It is interesting to note that in his proof of the Ryser--Brualdi--Stein conjecture~\cite{MontRBS}, Montgomery also builds up his local absorbers using swapping gadgets, which he calls ``switchers.'' Such swapping gadgets are also used in~\cite{BGKLMO}, where they are called ``transformers.''} these `special' vertices and colors for arbitrary vertices and colors.
\item Our global absorbers are $\binom{V}{\ep n}$-absorbers and $\binom{\C}{\ep n}$-absorbers for an $\Omega(n)$-size set $V$ of vertices or $\C$ of colors. These global absorbers allow us a lot of flexibility to absorb whichever vertices/colors we need; they are constructed in \Cref{sec:global_abs}. Our global absorbers are built from several local absorbers according to the template of the distributive absorption method due to Montgomery~\cite{Mont}.
\end{itemize}

In~\cite{KLS}, the gadgets were found with some fairly ad hoc arguments using Szemer\'edi's Regularity Lemma. While we believe such an approach would also work for our gadgets, we avoid the poor parametric dependencies caused by the Regularity Lemma and instead use a more general framework involving Suen's inequality (\Cref{thm:suen}), which is a generalization of Janson's inequality appropriate in our setting of randomly colored digraphs.

There are some simple heuristics to use when determining whether a given gadget can be found robustly in a randomly perturbed digraph. The following is not a precise or exhaustive list, but gives some idea of how one can narrow down the scope of what gadgets might be good candidates for this problem.
\begin{itemize}
\item The subdigraph induced by the deterministic edges must form a subdigraph of a blow-up of a consistently oriented path, otherwise they cannot be guaranteed to appear; see \Cref{sec:prelim:supersat}.
\item The subdigraph induced by the random edges must be a forest, since cycles are rare in random digraphs of the edge-density we consider.
\item Not too many colors can be repeated, as these are unlikely events; this heuristic is clearly expressed in our use of Suen's inequality.
\item Random edges are too rare in order to hit a given color reliably, so if a given color is specified, it must be on a deterministic edge.
\end{itemize}

We package our use of randomness, either from the random edges or the random colors, in various pseudorandom conditions, so that we are explicit throughout about how the randomness is being used. Most of our pseudorandom conditions are laid out in \Cref{def:pseudo}, with one additional complex condition in \Cref{def:cgpseudorandom}.

\section{Preliminaries}\label{sec:prelim}

In this section, we gather some basic lemmas concerning random graphs, random colorings, supersaturation, and probability.

\subsection{Probabilistic Tools}

To find some of our absorbing gadgets, we first show that there are many copies of the desired digraph, and then find one that is appropriately colored. We do this second step with Suen's inequality, which is a generalization of Janson's inequality.

\begin{thm}[Suen's Inequality (see~\protect{\cite[Theorem 3]{Suen}})]\label{thm:suen}
Let $\{I_i\}_{i \in [m]}$ be a family of indicator random variables, and let $\Gamma$ be a \emph{dependency graph} for $\{I_i\}_{i \in [m]}$, which is a graph whose vertex set is $[m]$ and where $\{I_i\}_{i \in A}$ and $\{I_i\}_{i \in B}$ are independent whenever there is no edge between $A \subseteq [m]$ and $B \subseteq [m]$ in $G$. Let $\mu = \sum_i \mathbb{P}(I_i = 1)$, $\delta = \max_i \sum_{j \sim i} \mathbb{P}(I_j=1)$, and $\Delta = \sum_{i \sim j} \mathbb{P}(I_i = I_j = 1)$, where $\sim$ represents adjacency in $\Gamma$ and the sum in $\Delta$ is over unordered pairs. Then
\[ \mathbb{P}\left( I_i = 0 \text{ for all } i \in [m] \right) \leq \exp\left( -\min\left( \frac{\mu^2}{8\Delta}, \frac{\mu}{6\delta}, \frac{\mu}{2} \right) \right) .\]
\end{thm}

\subsection{Pseudorandomness}

With one notable exception (see \Cref{def:cgpseudorandom}), we access the randomness in the edges and colors via the following pseudorandom condition. This condition and the following lemmas which use it are similar to those used in~\cite{ABKPT,KLS}.

\begin{definition}[Edge-Color Pseudorandom]\label{def:pseudo}
An edge-colored digraph $D$ is \emph{$(n,\alpha,\ep)$-edge-color pseudorandom} if the following hold:
\begin{enumerate}
\item\label{pseudo:mindeg} $D$ has $n$ vertices, $n$ colors, and $\delta^0(D) \geq \alpha n$.
\item \label{pseudo:edgecolor} For all $U,W\sbeq V(D)$ with $|U|,|W|\geq \ep n$, $\C'\sbeq \C(D)$ with $|\C'|\geq \ep n$, and $\star\in\arcset$, there exists an edge $e\in E^{\star}(U,W)$ with $\C(e)\in\C'$.
\item \label{pseudo:maxdeg} For all $c\in\C$, $v\in V(D)$, $\star \in \arcset$, $d^\star_c(v)\leq \frac{n}{\log n}$.
\item \label{pseudo:match} For all $c \in \C$, there exists a matching of size $\alpha n/3$ in $E_c(D)$.
\item \label{pseudo:avgedges} For all $c\in\C$, $|E_c(D)|\leq 10n$.
\item \label{pseudo:colordiversity} For all $v\in V(D)$ and $\star \in \arcset$, $|\C^\star(v)|\geq \alpha n/2$.
\end{enumerate}
\end{definition}

Naturally, our randomly colored randomly perturbed digraph is pseudorandom with high probability.

\begin{lemma}\label{lem:ec-pseudorandom}
Let $0 < \ep \ll \alpha$, and let $C > \frac{3+3\log(1/\ep)}{\ep^2}$.
Let $D_0$ be an $n$-vertex digraph with minimum semi-degree $\delta^0(D_0) \geq \alpha n$. Then $D = D_0 \cup D(n,C/n)$, uniformly edge-colored from $[n]$, is $(n, \alpha,\ep)$-edge-color-pseudorandom with high probability.
\end{lemma}
\begin{proof}
For \Cref{pseudo:mindeg}, $D$ has $n$ vertices by definition, and $\delta^0(D) \geq \alpha n$ is immediate from $\delta^0(D_0) \geq \alpha n$. The probability that $\C(D)$ does not contain a particular color $c \in [n]$ is $(1-1/n)^{|E(D)|} \leq \exp\left( -\alpha n \right)$. Thus, by a union bound, $\C(D) = [n]$ with high probability.

For \Cref{pseudo:edgecolor}, fix $U,W \subseteq V(D)$ and $\D \subseteq \C(D)$ with $|U| = |W| = |\D| = \ep n$, and fix $\star \in \arcset$. Let $A_{U,W,\D}$ be the event that $\C(E^\star(U,W))$ is disjoint from $\D$. Considering the random edges of $D$ (and accounting for the fact that $U$ and $W$ may not be disjoint), we have
\[ \mathbb{P}(A_{U,W,\D}) \leq \left(1-\frac{C}{n}\cdot \frac{|\D|}{n}\right)^{|U|\cdot|W|-\max\{|U|,|W|\}} \leq \exp\left( - C\ep^3 n + C\ep^2 \right) .\]
After a union bound over all choices of $U$, $W$, $\D$, and $\star$, we have that the probability that \Cref{pseudo:edgecolor} fails is
\[ \mathbb{P}\left(\bigcup_{U,W,\D}A_{U,W,\D}\right) \leq 2\binom{n}{\ep n}^3 \exp(-C\ep^3n+C\ep^2) \leq 2\exp\left( 3\ep n \log\frac{e}{\ep} - C \ep^3 n + C \ep^2 \right) ,\]
which is $o(1)$ with $C$ as specified in the statement of the lemma, where we use the bound $\binom{a}{b} \leq (ae/b)^b$.

Next, we check \Cref{pseudo:maxdeg}. Let $c\in\C$, $v\in V(D)$, and $\star \in \arcset$. Let $A\sbeq V(D)$ be of size $|A|=n/\log n$. Then the probability that $A \subseteq N_c^\star(v)$ is at most 
\[ n^{-n/\log n} = \exp(-n) ,\]
so we may take a union bound over all $c$, $v$, $\star$, and $A$ to get that the probability \Cref{pseudo:maxdeg} fails is at most
\[ 2n^2 \binom{n}{n/\log n} e^{-n} \leq \exp\left( \log(2n^2) + \frac{n}{\log n} \log(e\log n) - n \right) = o(1) .\]

Next, we check \Cref{pseudo:match}. Fix $c \in \C$, and let $M$ be a maximum matching in color $c$. Note that $S = V(D) \setminus V(M)$ is an independent set in color $c$. If $M$ has size less than $\alpha n/3$, then $|S| \geq (1-2\alpha/3)n$ and every vertex of $S$ has at least $\alpha n/3$ out-neighbors in $S$, since $\delta^0(D) \geq \alpha n$. Thus $S$ spans at least $\alpha n^2/9$ edges. The probability that $c$ does not appear on any of these edges is at most
\[ \left( 1 - \frac{1}{n} \right)^{\alpha n^2/9} \leq \exp\left( -\alpha n/9 \right) .\]
Thus we can take a union bound over all colors $c$ to obtain \Cref{pseudo:match} with high probability.

\Cref{pseudo:avgedges} easily follows from the Chernoff bound. Since $D$ has at most $n^2$ edges, the expected number of edges of color $c$ is at most $n$, and hence with probability exponentially small in $n$, there are at most $10n$ edges of color $c$. With a union bound, \Cref{pseudo:avgedges} holds.

Lastly, we check \Cref{pseudo:colordiversity}. Fix $v \in V(D)$, $\star \in \arcset$, and $\C' \subseteq \C$ with $|\C'| = \alpha n/2$. Since $\delta^0(D) \geq \alpha n$, the probability that every $\star$-edge incident to $v$ has color in $\C'$ is at most
\[ \left( \frac{|\C'|}{n} \right)^{\alpha n} \leq (\alpha/2)^{\alpha n} .\]
Taking a union bound over all $v$, $\star$, and $\C'$, we have that the probability \Cref{pseudo:colordiversity} fails is at most
\[ 2n \binom{n}{\alpha n/2} (\alpha/2)^{\alpha n} \leq 2n \left( \frac{e\alpha}{4} \right)^{\alpha n/2} = o(1) .\]

Since each of the pseudorandom conditions hold with high probability, they all hold simultaneously with high probability.
\end{proof}

An important consequence of \Cref{def:pseudo} is that such graphs are robustly well connected in a rainbow sense: between any two vertices, we can find a rainbow path of length $3$, avoiding a small set of vertices and colors.

\begin{lemma}[Connecting Lemma]\label{lem:connecting}
Let $\frac{1}{n} \ll \ep \ll \nu \ll \alpha \leq 1$.
If $D$ is an $(n, \alpha, \ep)$-edge-color pseudorandom digraph, then for all $v,w\in V(D)$, $\star\in\arcset^3$, $V'\sbeq V(D)$, $\C'\sbeq \C(D)$ satisfying $|V'|,|\C'|\geq (1-\nu)n$, there exists a rainbow path $v u_1 u_2 w$ with orientation $\star$ with internal vertices in $V'$ and colors in $\C'$.
\end{lemma}

\begin{proof}
By \Cref{def:pseudo} (\ref{pseudo:colordiversity}), we have that $|\C^{\star_1}(v)|\geq \alpha n/2$ and $|\C^{-\star_3}(w)|\geq \alpha n/2$. Let $\C_1 \subseteq \C^{\star_1}(v)$ and $\C_2 \subseteq \C^{-\star_3}(w)$ be disjoint sets of size $\alpha n/4$, and let $U = N^{\star_1}_{\C_1}(v) \cap V'$ and $W = N^{-\star_3}_{\C_2}(w) \cap V'$, so that $|U| \geq |C_1| - \nu n \geq \ep n$ and $|W| \geq |\C_2| - \nu n \geq \ep n$. Let $\D=\C'\setminus (\C_1\cup \C_2)$, so that $|\D|\geq (1-\nu-2\ep)n \geq \ep n$. Therefore by \Cref{def:pseudo} (\ref{pseudo:edgecolor}), there exists $(u_1u_2)^{\star_2}\in E^{\star_2}(U,W)$ such that $ \C((u_1u_2)^{\star_2})\in \D$. Thus $v u_1 u_2 w$ is a rainbow path with orientation $\star$, internal vertices in $V'$, and colors in $\C'$.
\end{proof}

Also needed is an extension of \Cref{lem:connecting} where we can choose the lengths of the paths.

\begin{lemma}[Rainbow Path Lemma]\label{lem:shortpaths}
Let $\frac{1}{n} \ll \ep \ll \alpha \leq 1$.
Let $k$ be a positive integer, and let $\star \in \arcset^{k-1}$.
Let $D$ be an $(n,\alpha,\ep)$-edge-color-pseudorandom digraph, let $\D \sbeq \C(D)$ with $|\D| \geq k\ep n$, and let $A_1,\ldots, A_k\sbeq V(D)$ be pairwise disjoint with $|A_i|\geq 2\ep n$.
Then there is a rainbow path $v_1 v_2 \ldots v_k$ with $v_i\in A_i$ for $i \in [k]$ and $(v_iv_{i+1})^{\star_i}\in E(D)$ and $\C((v_iv_{i+1})^{\star_i})\in \D$ for $i \in [k-1]$.
\end{lemma}

\begin{proof}
Define a \emph{valid path} as a path which can be extended to a path described in the lemma, that is, a path $v_1v_2\ldots v_j$ for some $j \in [k]$ such that $v_i\in A_i$ for all $i \in [j]$ and $(v_iv_{i+1})^{\star_i}\in E(D)$ and $\C((v_iv_{i+1})^{\star_i})\in \D$ for $i \in [j-1]$.

We prove by induction on $j$ that there are at least $\ep n$ valid paths from $A_1$ to $A_j$ with distinct endpoints. The base case is trivial. For the inductive step, suppose the claim holds for $j \in [k-1]$. By the inductive hypothesis, there exists a set $B_j\sbeq A_j$ of endpoints of valid paths with $|B_j|=\ep n$. Fix $t$ valid paths, one ending at each vertex of $B_j$, and let $\D_j$ be the set of colors used in these paths, so that $|\D_j| \leq (j-1) \ep n$.

Let $W \subseteq A_{j+1}$ be the set of vertices which are not the endpoints of valid paths, and let $\C' = \D \setminus \D_j$, so that $|\C'| \geq k\ep n - (j-1) \ep n \geq \ep n$. By \Cref{def:pseudo} (\ref{pseudo:edgecolor}), if $W$ had at least $\ep n$ vertices, then there would be an edge in $E^{\star_j}(B_j, W)$ with color in $\C'$, thus extending a valid path, contradicting the definition of $W$. Thus $|W| \leq \ep n$, and so there are at least $\ep n$ valid paths from $A_1$ to $A_{j+1}$ with distinct endpoints.
\end{proof}

The last application of \Cref{def:pseudo} is to finding almost spanning rainbow paths. The proof of the following lemma uses depth-first search and is similar to an uncolored undirected version (see~\cite{KrivCycles}); we give the proof here for completeness. 

\begin{lemma}[Long Path Lemma]\label{lem:longpath}
Let $\frac{1}{n} \ll \ep \ll \nu \ll \alpha \leq 1$.
Let $D$ be an $(n,\alpha,\ep)$-edge-color pseudorandom digraph, and let $k<(1-\nu-3\ep)n$. Then for any $V' \subseteq V(D)$ and $\C' \subseteq [n]$ with $|V'|, |\C'| \geq (1-\nu) n$ and for any $\star \in \arcset^{k}$, there exists a rainbow path in $D$ with orientation $\star$, vertices in $V'$, and colors in $\C'$.
\end{lemma}

\begin{proof}
We proceed by depth-first search: at each step in the process, we will have a partition $V(P) \cup S \cup U$ of $V'$, where $P$ is our current path, $S$ is the set of `searched' vertices that were previously searched, and $U$ is the set of `unprocessed' vertices that we have not yet searched. Vertices only move from $U$ to $P$ and from $P$ to $S$.

Whenever an unprocessed vertex goes into the path, we record the index of the vertex in the path it acted as, implicitly defining a partial function $f:V'\to[k+1]$ as we go along; for example, the first vertex $v_1$ we process satisfies $f(v_1)=1$. We also record all colors ever used on $P$ in the set $\D \subseteq \C'$.

We initialize the partition as $P=\emptyset$, $S=\emptyset$, and $U=V'$. For each step, if $|P|=0$, then choose $v\in U$ arbitrarily and move $v$ from $U$ to $P$, recording $f(v) = 1$. If $|P|>0$, then let the most recently added vertex to $P$ be called $w$, acting as index $j = f(w)$ of the path. If $N_{\C'\setminus \D}^{\star_j}(w)\cap U$ is nonempty, then pick $u \in N_{\C'\setminus \D}^{\star_j}(w)\cap U$ arbitrarily, move $u$ from $U$ to $P$, record $f(u) = j+1$, and add $\C((wu)^{\star_j})$ to $\D$. Otherwise, if $N_{\C'\setminus \D}^{\star_j}(w)\cap U$ is empty, move $w$ from $P$ to $S$.

The algorithm terminates if $P$ has length $k$ or $|U|=0$. The algorithm must terminate because the vertices only move from $U$ to $P$ and from $P$ to $D$, so eventually $U = \emptyset$, unless the algorithm terminated with a path of length $k$.

Assume for sake of contradiction the algorithm terminates with $|V(P)|\leq k$. Then $|U|=0$ and $|S|\geq (1-\nu)n-k \geq 3\ep n$ at termination. Consider the first stage where $|S|=2\ep n$. Since the algorithm did not terminate with $|V(P)|= k+1$, we must have at this stage $|V(P)|\leq k$, and thus $|U| = |V'| - |V(P)| - |S| \geq \ep n$.

If $s\in S$ has was added to the path at index $i$, then $N_{\C'\setminus \D}^{\star_i}(s)\cap U=\emptyset$ by definition of $S$, since $\D$ only accumulates colors. Also observe that $|\D| \leq |S| + |V(P)| < (1-\nu-\ep)n$.

Since $|S| = 2\ep$, one of $\{s\in S\mid \star_{f(s)}=\rightarrow\}\dot \cup \{s\in S\mid \star_{f(s)}=\leftarrow\}$ must have size at least $\ep n$; without loss of generality, suppose $W=\{s\in S\mid \star_{f(s)}=\rightarrow\}$ has size at least $\ep n$. Then by \Cref{def:pseudo} (\ref{pseudo:edgecolor}), there must be an edge in $E^{\rightarrow}(W,U)$ with color in $\C'\setminus \D$, since $|\C'\setminus \D| \geq\ep n$, contradicting the definition of $S$. 

Thus the algorithm must terminate with $P$ a path of length $k$. This path has orientation $\star$ by construction, and it is rainbow because we only used edges with colors not previously chosen at any stage.
\end{proof}

\subsection{Supersaturation}\label{sec:prelim:supersat}

Supersaturation is the phenomenon whereby when a graph (or digraph or hypergraph) is much more than dense enough to contain a single copy of a substructure, it actually has many copies. For example, the K\H{o}v\'ari--S\'os--Tur\'an Theorem states that if an $n$-vertex graph $G$ has more than $C_{s,t} n^{2-1/s}$ edges, then $G$ contains a copy of the complete bipartite graph $K_{s,t}$. A standard supersaturation form of this theorem is that if $G$ is an $n$-vertex graph with $\alpha n^2$ edges, then $G$ contains $\Omega_\alpha(n^{s+t})$ copies of $K_{s,t}$, and as a consequence, $G$ contains $\Omega_\alpha(n^{|V(H)|})$ copies of every (fixed size) bipartite graph $H$. Note that we may only have such a statement for all $\alpha>0$ if $H$ is bipartite, since $G$ could be a balanced complete bipartite graph. We need the following directed version of this supersaturation statement.

\begin{lemma}\label{lem:supersat}
Let $H$ be a subdigraph of a blow-up of a consistently oriented path. For every $\alpha>0$, there exists $\theta > 0$ such that if $D$ is an $n$-vertex digraph with $\delta^+(D) \geq \alpha n$, then $D$ contains at least $\theta n^{|V(H)|}$ copies of $H$.
\end{lemma}

Just as we cannot hope for a supersaturation statement for graphs which are not bipartite, we cannot generalize \Cref{lem:supersat} to $H$ which are not subgraphs of blow-ups of a consistently oriented path. This can be seen by considering the blow-up of a sufficiently long consistently oriented cycle. We also note that we cannot weaken the minimum out-degree condition in \Cref{lem:supersat} to a condition on the number of edges since a balanced complete bipartite digraph with all edges oriented from one part to the other does not have a consistently oriented path of length $2$.

To prove \Cref{lem:supersat}, we use a supersaturation form of the hypergraph K\H{o}v\'ari--S\'os--Tur\'an Theorem, whose proof is standard.

\begin{thm}\label{thm:hyperKST}
Let $k \geq 2$ be an integer.
For every $\beta > 0$, there exists $\theta > 0$ such that the following holds.
Let $\mathcal{H}$ be a $k$-uniform hypergraph with $\beta n^k$ edges. Then $\mathcal{H}$ contains $\theta n^{s_1+\cdots+s_k}$ copies of $K^{(k)}_{s_1,\dots,s_k}$, the complete $k$-uniform $k$-partite hypergraph with part sizes $s_1, \dots, s_k$.
\end{thm}

We note that $\theta$ depends polynomially on $\beta$ in \Cref{thm:hyperKST}.

\begin{proof}[Proof of \Cref{lem:supersat}]
It suffices to prove the lemma for the blow-up of a consistently oriented path. Let $H$ be such a blow-up whose parts have size $s_1, \dots, s_k$, in that order. We will show that the number of copies of $H$ in $D$ is $\Omega_\alpha(n^{s_1+\cdots+s_k})$.

First, randomly partition $V(D)$ into $k$ parts $V_1, \dots, V_k$. By the Chernoff bound, every $v \in V_i$ has degree at least $\alpha n/2k$ to $V_{i+1}$ with high probability, and each $|V_i| \geq n/2k$ with high probability. Thus there are at least $\alpha^{k-1} (n/2k)^k$ oriented paths following this partition.

Form a $k$-partite $k$-uniform hypergraph $H$ with parts $V_1, \dots, V_k$ whose hyperedges are these oriented paths. By \Cref{thm:hyperKST}, there exists $\Omega_\alpha(n^{s_1+\cdots+s_k})$ many complete $k$-partite subhypergraphs of $H$ with parts sizes $s_1, \dots, s_k$, each of which translates to a copy of $H$ in $D$.
\end{proof}

\section{Gadgets}\label{sec:gadgets}

Recall \Cref{def:absorber} of an $\mathcal{S}$-absorber. 
What we refer to as `gadgets' are $\{c_1,c_2\}$-absorbers or $\{v_1,v_2\}$-absorbers for a pair of colors $c_1, c_2$ or a pair of vertices $v_1, v_2$.  That is, these absorbers absorb exactly one of a pair of colors or vertices. These are critical ingredients in our more complicated absorber constructions.

\subsection{Color gadget}

Our first gadget is a $\{c_1,c_2\}$-gadget for two colors $c_1$ and $c_2$, inspired by the gadget from~\cite{KLS}, which in turn is inspired by the (undirected) gadget in~\cite{GKKO}. In~\cite{KLS}, Szemer\'edi's regularity lemma and some probabilistic arguments were used to find these gadgets. While this approach could likely work for us, we sidestep the complications and poor parametric dependencies of the regularity lemma and give a proof of existence via supersaturation (\Cref{lem:supersat}) and an application of Suen's inequality (\Cref{thm:suen}).

We define a peculiar pseudorandom property which simply states that these color gadgets exist robustly. Then we show that a randomly colored randomly perturbed digraph satisfies this pseudorandom property with high probability. We do this so that any time we need these color gadgets, all we need to do is invoke the pseudorandom property.

\begin{definition}[Color-Gadget Pseudorandom]\label{def:cgpseudorandom}
An $n$-vertex edge-colored digraph $D$ is \emph{$\nu$-color-gadget pseudorandom} if for every  $V'\subseteq V(D)$ and $\C'\subseteq \C(D)$ such that $|V'|,|\C'|\geq (1-\nu)n$, for every $c_1,c_2\in \C(D)$, and for every $\star \in \arcset^7$, there exists a $(\{c_1,c_2\},\star)$-absorber with internal vertices in $V'$ and internal colors in $\C'$.
\end{definition}

\begin{lemma}[Color-Gadget Lemma]\label{lem:cg-pseudorandom}
Let $\frac{1}{n} \ll \frac{1}{C} \ll \nu \ll \alpha \leq 1$.
Let $D_0$ be an $n$-vertex digraph with minimum semi-degree $\delta^0(D_0) \geq \alpha n$.
Then $D=D_0\cup D(n,C/n)$, uniformly edge-colored from $[n]$, is $\nu$-color-gadget pseudorandom with high probability.
\end{lemma}

\begin{proof}
Our color gadgets will take the form depicted in \Cref{fig:colorgadget}. We first show that there are many copies of the underlying directed $K_{2,4}$ in $D_0$, then we use Suen's inequality to find an appropriately colored copy. Finally, we find the connecting path between $w_2$ and $w_3$ in \Cref{fig:colorgadget} using \Cref{lem:connecting}.

Fix $V'\sbeq V$ and $\C'\sbeq \C$, each of size $(1-\nu)n$, fix $\star \in \arcset^7$, and fix $c_1,c_2 \in \C(D)$. We will show that with very high probability, there exists a $(\{c_1,c_2\},\star)$-absorber with vertices in $V'$ and internal colors in $\C'$. Taking a union bound over all choices of $V', \C', \star, c_1, c_2$ will give the result.

Let $K=(v_1,v_2,w_1,w_2,w_3,w_4)\in V(D)^6$. We say that $K$ is a \emph{$K_{2,4}$ with orientation $\star$} if $(w_1v_i)^{\star_1}, (v_iw_2)^{\star_2}, (w_3v_i)^{\star_6}, (v_iw_4)^{\star_7} \in E(D)$ for both $i\in\{1,2\}$. We define $D[K]$ to be the subdigraph of $D$ on vertex set $\{v_1,v_2,w_1,w_2,w_3,w_4\}$ with only these edges specified. 

Arbitrarily partition $\C'\setminus\{c_1,c_2\}$ into sets $\C_1, \C_2, \C_3$ with $|\C_i| = n/4$ for all $i \in [3]$. We say that $K$ has a \emph{valid coloring with colors $(a,b,d)$} if $\C((w_1v_i)^{\star_1})=c_i$, $\C((v_i w_2)^{\star_2}) = a\in \C_1$, $\C((w_3v_i)^{\star_6}) = b\in \C_2$, and $\C((v_iw_4)^{\star_7}) = d\in\C_3$ for both $i \in \{1,2\}$. See \Cref{fig:colorgadget}.

Observe that $K_{2,4}$s with orientation $\star$ are blow-ups of a consistently oriented path. We apply \Cref{lem:supersat} to the digraph $D[V']$, which has minimum semi-degree at least $(\alpha-\nu)n \geq \alpha n/2$, obtaining a set $\K \sbeq V(D)^6$ of $K_{2,4}$'s with orientation $\star$, whose vertices are in $V'$, of size $|\K| \geq \theta n^6$, where $\theta$ depends on $\alpha$.

For each $K \in \K$, let $I_K$ be the indicator random variable for $K$ having a valid coloring. Let $\Gamma$ be the dependency graph on $\K$ with an edge connecting $K$ and $K'$ if $D[K]$ and $D[K']$ share an edge; we represent this as $K \sim K'$. Since the edges of $D$ receive colors independently, $\Gamma$ is a dependency graph. We now compute $\mu$, $\delta$, and $\Delta$ from \Cref{thm:suen}.

For each $K \in \K$, note that
\begin{equation}\label{eq:suen:prob}
\mathbb{P}(I_K=1) = \frac{(n/4)^3}{n^8}=\frac{1}{64}n^{-5} ,
\end{equation}
so that
\begin{equation}\label{eq:suen:mu}
\mu = \sum_{K\in\K}\mathbb{P}(I_K=1) \geq \theta n^6\cdot\frac{1}{64} n^{-5} =\frac{\theta}{64} n .
\end{equation}
Since $K' \sim K$ if and only if $D[K']$ and $D[K]$ share an edge, there are at most $8n^4$ neighbors of $K$ in $\Gamma$. Thus by~\eqref{eq:suen:prob},
\begin{equation}\label{eq:suen:delta}
\delta=\max_{K\in\K}\sum_{K'\sim K}\mathbb{P}(I_{K'}=1) \leq 8 n^4\cdot \frac{1}{64}n^{-5} =\frac{1}{8}n^{-1}.
\end{equation}
Lastly, to compute $\Delta$, let $X_{(a,b,d)}(K)$ be the event that $K \in \K$ has a valid coloring with colors $(a,b,d)$. We first observe that
\begin{align*}
    \Delta &= \sum_{K\sim K'}\mathbb{P}(I_K=I_{K'}=1)\\
    &= \frac{1}{2}\sum_{K\in \K}\mathbb{P}(I_K=1)\sum_{K'\sim K}\mathbb{P}(I_{K'}=1\mid I_K=1)\\
    &\leq \frac{1}{2}\sum_{K\in\K}\mathbb{P}(I_K=1)\max_{K_0\in\K}\sum_{K'\sim K_0}\mathbb{P}(I_{K'}=1\mid I_{K_0}=1)\\
    &=\frac{1}{2}\mu \max_{K_0\in\K}\sum_{K'\sim K_0}\mathbb{P}(I_{K'}=1\mid I_{K_0}=1) \\
    &=\frac{1}{2}\mu \max_{K_0\in\K}\sum_{K'\sim K_0}\sum_{(a,b,d) \in \C_1 \times \C_2 \times \C_3} \mathbb{P}(I_{K'}=1 \mid X_{(a,b,d)}(K_0)) \cdot \mathbb{P}(X_{(a,b,d)}(K_0) \mid I_{K_0} = 1) \\
    &\leq \frac{1}{2}\mu \max_{K_0\in\K} \max_{(a,b,d) \in \C_1\times \C_2 \times \C_3} \sum_{K'\sim K_0} \mathbb{P}(I_{K'}=1 \mid X_{(a,b,d)}(K_0)) .
\end{align*}
Fix $K_0=(v_1,v_2,w_1,w_2,w_3,w_4) \in \K$ and $(a,b,d) \in \C_1 \times \C_2 \times \C_3$. We aim to show that $\sum_{K'\sim K_0} \mathbb{P}(I_{K'}=1 \mid X_{(a,b,d)}(K_0)) = O(1)$, which implies that $\Delta = O(\mu)$, where the $O(\cdot)$ notation hides an absolute constant not depending on any of our parameters. We partition $\{K'\in\K\mid K'\sim K_0\}$ based on which edges $D[K_0]$ and $D[K']$ have in common. For each $A\sbeq E(D[K_0])$, let $\K_A=\{K'\in\K\mid E(D[K_0])\cap E(D[K'])=A\}$, and let
\[ \Delta_A = \sum_{K'\in \K_A} \mathbb{P}(I_{K'}=1 \mid X_{(a,b,d)}(K_0)) .\]
Since there are $O(1)$ many choices for $A$, it suffices to show that $\Delta_A = O(1)$ for every $A \neq \emptyset$. We omit the orientations on the edges of $A$ since they are given by $K_0$. Let $v(A)$ denote the number of vertices spanned by $A$, let $e(A)$ denote the number of edges in $A$, and let $\ell(A)$ be the number of $i$ such that $A$ contains an edge of a color in $\C_i$.

Observe that if $K' = (v_1',v_2',w_1',w_2',w_3',w_4')$ and $v_i w_j \in A$, then $w_j = w_j'$ since $c_1, c_2$ are specified and $\C_1, \C_2, \C_3$ are disjoint, but it is possible that $v_i = v_i'$ or $v_i = v_{3-i}'$. This means that $|\K_A| \leq 2n^{6-v(A)}$, that is, there is at most $2n^{6-v(A)}$ choices for $K' \in \K_A$: a factor of $n$ for each vertex of $K'$ not in $K$, and a factor of two for possibly swapping $v_1$ and $v_2$. There are at most $n^{3-\ell(A)}$ choices for the remaining colors of $K' \in \K(A)$ to produce a valid coloring, and all of these colors appear on $K'$ with probability $n^{-8+e(A)}$ when conditioned on $X_{(a,b,d)}(K_0)$. Thus
\begin{equation}\label{eq:DeltaA}
\Delta_A \leq 2n^{6-v(A)} \cdot n^{3-\ell(A)} \cdot n^{-8+e(A)} = 2n^{1-(v(A)+\ell(A)-e(A))} .
\end{equation}
To show that $\Delta_A = O(1)$, it suffices to show that $v(A)+\ell(A)-e(A) \geq 1$ for all choices of $A \neq \emptyset$. We consider various cases based on $v(A)-e(A)$. 
\begin{enumerate}
\item $v(A)-e(A)\geq 1$. Thus $v(A)+\ell(A)-e(A) \geq 1$ trivially. 
\item $v(A)-e(A)=0$. Then there is at least one cycle in $A$, so there must be an edge with a color in some $\C_i$, hence $\ell(A)\geq 1$. Thus $v(A)+\ell(A)-e(A)\geq 1$.
\item $v(A)-e(A)=-1$. Then there are at least two cycles in $A$, so there must be edges with colors in $\C_i$ for two distinct values of $i$, hence $\ell(A)\geq 2$. Thus $v(A)+\ell(A)-e(A)\geq 1$. 
\item $v(A)-e(A)\leq -2$. Since every four vertices of $K_{2,4}$ spans at most four edges, and every five vertices of $K_{2,4}$ spans at most six edges, we have that $v(A) = 6$ and $e(A) = 8$. Thus $\ell(A) = 3$, and $v(A)+\ell(A)-e(A) = 1$ as desired.
\end{enumerate}
Putting this all together, we have that $\Delta_A = O(1)$ for every $A$, and hence
\[ \Delta \leq O(\mu) .\]

By \Cref{thm:suen}, using~\eqref{eq:suen:mu} and~\eqref{eq:suen:delta}, the probability that there is no $K \in \K$ with a valid coloring is at most
\[ \exp\left( - \min\left( \frac{\mu^2}{8\Delta}, \frac{\mu}{6\delta}, \frac{\mu}{2} \right) \right) \leq \exp\left( - \min\left( \Omega(\mu), \Omega(n\mu), \frac{\mu}{2} \right) \right) = \exp\left( - \Omega(\theta n) \right) .\]

We now take a union bound over all choices of $V', \C', c_1, c_2, \star$: the probability that $D$ does not have a $K \in \K$ with a valid coloring is at most
\[ \binom{n}{\nu n}^2 n^2 2^7 e^{-\Omega(\theta n)} = o(1) ,\]
since $\nu \ll \alpha$, and $\theta$ depends only on $\alpha$ via \Cref{lem:supersat}. Thus with high probability, $D$ contains a $K = (v_1,v_2,w_1,w_2,w_3,w_4) \in \K$ with a valid coloring no matter the choice of $V'$ and $\C'$.

Finally, by \Cref{lem:ec-pseudorandom}, $D$ is $(\alpha,\ep)$-edge-color pseudorandom for some $\ep\ll\alpha$ (see~\Cref{def:pseudo}) with high probability, since $C$ is sufficiently large. Thus by \Cref{lem:connecting}, we can find a path $w_2 u_1 u_2 w_3$ of length $3$ between $w_2$ and $w_3$ with orientation $(\star_3,\star_4,\star_5)$, avoiding the vertices and colors in $V(D)\setminus V'$ and $\C(D)\setminus \C'$ and $K$. This yields our desired $\{c_1,c_2\}$-gadget pictured in \Cref{fig:colorgadget}. The $c_1$-absorbing path is $w_1v_1w_2u_1u_2w_3v_2w_4$ and the $c_2$-absorbing path is $w_1v_2w_2u_1u_2w_3v_1w_4$.
\end{proof}

\begin{figure}
\begin{center}
   \begin{tikzpicture}[
    >=Stealth,
    vertex/.style={circle,draw,inner sep=1.5pt},
    edgelabel/.style={midway, fill=white, inner sep=1pt, text=black}
]

\tikzset{
  pics/colorgadget/.style args={#1/#2/#3}{
    code={
      \draw (0,0) -- (.75,.75) -- (1.5,0) -- (.75,-.75) -- cycle;
      \draw (0.25,0) -- (1.25,0);
      \node[draw=orange,draw opacity =0,circle,inner sep=1pt] at (.75,0.3) {#1};
      \node[draw=purple,draw opacity =0,circle,inner sep=1pt] at (.75,-0.3) {#2};
      \node at (.75,-.95) {#3};
    }
  }
}

\node[vertex,label=above:$v_1$] (v1) at (2,2.6) {};
\node[vertex,label=below:$v_2$] (v2) at (2,0) {};
\node[vertex,label=left:$w_1$] (w1) at (0,1.3) {};
\node[vertex,label=left:$w_2$] (w2) at (1.3,1.3) {};
\node[vertex,label=right:$w_3$] (w3) at (2.7,1.3) {};
\node[vertex,label=right:$w_4$] (w4) at (4,1.3) {};

\draw[<-, orange] (v1) -- node[edgelabel] {$c_1$} (w1);
\draw[->, purple] (w1) -- node[edgelabel] {$c_2$} (v2);
\draw[->, green!60!black] (v1) -- node[edgelabel] {$d$} (w4);
\draw[<-, green!60!black] (w4) -- node[edgelabel] {$d$} (v2);

\draw[<-, red] (v1) -- node[edgelabel] {$a$} (w2);
\draw[->, red] (w2) -- node[edgelabel] {$a$} (v2);
\draw[->, blue] (v1) -- node[edgelabel] {$b$} (w3);
\draw[<-, blue] (w3) -- node[edgelabel] {$b$} (v2);

\draw[dotted] (w2) -- 
    node[edgelabel, below] {$\scriptstyle \star_3\star_4\star_5$} 
    (w3);

\begin{scope}[xshift=-1cm]

\node at (-0.4,1.3) {$:=$};

\pic at (-2.2,1.3) {colorgadget={$c_1$/$c_2$/$\scriptstyle\star_1\ldots \star_7$}};

\end{scope}

\end{tikzpicture}

\hrule

\begin{tikzpicture}[
    >=Stealth,
    vertex/.style={circle,draw,inner sep=1.5pt},
    edgelabel/.style={midway, fill=white, inner sep=1pt, text=black}
]

\tikzset{
  pics/colorgadget/.style args={#1/#2/#3}{
    code={
      \draw (0,0) -- (.75,.75) -- (1.5,0) -- (.75,-.75) -- cycle;
      \draw (0.25,0) -- (1.25,0);

      \node[draw=orange,circle,inner sep=1pt] at (.75,0.3) {#1};

      \node at (.75,-0.3) {#2};
      \node at (.75,-.95) {#3};
    }
  }
}

\node[vertex,label=above:$v_1$] (v1) at (2,2.6) {};
\node[vertex,label=below:$v_2$] (v2) at (2,0) {};
\node[vertex,label=left:$w_1$] (w1) at (0,1.3) {};
\node[vertex,label=left:$w_2$] (w2) at (1.3,1.3) {};
\node[vertex,label=right:$w_3$] (w3) at (2.7,1.3) {};
\node[vertex,label=right:$w_4$] (w4) at (4,1.3) {};

\draw[->, purple, opacity=0.1] (w1) -- (v2);
\draw[->, red, opacity=0.1] (w2) -- (v2);
\draw[->, blue, opacity=0.1] (v1) -- (w3);
\draw[->, green!60!black, opacity=0.1] (v1) -- (w4);

\draw[<-, orange] (v1) -- node[edgelabel] {$c_1$} (w1);
\draw[<-, green!60!black] (w4) -- node[edgelabel] {$d$} (v2);
\draw[<-, red] (v1) -- node[edgelabel] {$a$} (w2);
\draw[<-, blue] (w3) -- node[edgelabel] {$b$} (v2);

\draw[dotted] (w2) -- 
    node[edgelabel, below] {$\scriptstyle \star_3\star_4\star_5$} 
    (w3);

\begin{scope}[xshift=-1cm]
\node at (-0.4,1.3) {$:=$};
\pic at (-2.2,1.3) {colorgadget={$c_1$/$c_2$/$\scriptstyle\star_1\ldots\star_7$}};
\end{scope}

\end{tikzpicture}
\vrule\hspace{1mm}
\begin{tikzpicture}[
    >=Stealth,
    vertex/.style={circle,draw,inner sep=1.5pt},
    edgelabel/.style={midway, fill=white, inner sep=1pt, text=black}
]

\tikzset{
  pics/colorgadget/.style args={#1/#2/#3}{
    code={
      \draw (0,0) -- (.75,.75) -- (1.5,0) -- (.75,-.75) -- cycle;
      \draw (0.25,0) -- (1.25,0);

      \node at (.75,0.3) {#1};

      \node[draw=purple,circle,inner sep=1pt] at (.75,-0.3) {#2};

      \node at (.75,-.95) {#3};
    }
  }
}

\node[vertex,label=above:$v_1$] (v1) at (2,2.6) {};
\node[vertex,label=below:$v_2$] (v2) at (2,0) {};
\node[vertex,label=left:$w_1$] (w1) at (0,1.3) {};
\node[vertex,label=left:$w_2$] (w2) at (1.3,1.3) {};
\node[vertex,label=right:$w_3$] (w3) at (2.7,1.3) {};
\node[vertex,label=right:$w_4$] (w4) at (4,1.3) {};

\draw[<-, orange, opacity=0.1] (v1) -- (w1);
\draw[<-, red, opacity=0.1] (v1) -- (w2);
\draw[<-, blue, opacity=0.1] (w3) -- (v2);
\draw[<-, green!60!black, opacity=0.1] (w4) -- (v2);

\draw[->, purple] (w1) -- node[edgelabel] {$c_2$} (v2);
\draw[->, green!60!black] (v1) -- node[edgelabel] {$d$} (w4);
\draw[->, red] (w2) -- node[edgelabel] {$a$} (v2);
\draw[->, blue] (v1) -- node[edgelabel] {$b$} (w3);

\draw[dotted] (w2) -- 
    node[edgelabel, below] {$\scriptstyle \star_3\star_4\star_5$} 
    (w3);

\begin{scope}[xshift=-1cm]
\node at (-0.4,1.3) {$:=$};
\pic at (-2.2,1.3) {colorgadget={$c_1$/$c_2$/$\scriptstyle\star_1\ldots\star_7$}};
\end{scope}

\end{tikzpicture}
\end{center}
\caption{$(\{c_1,c_2\},\star)$-absorber with $\star=(\rightarrow,\leftarrow,\star_3,\star_4,\star_5,\leftarrow,\rightarrow)$ and $c_1$-absorbing and $c_2$-absorbing paths shown. We use the short-hand notation for this gadget given on the left side of the figure in later constructions.}
\label{fig:colorgadget}
\end{figure}
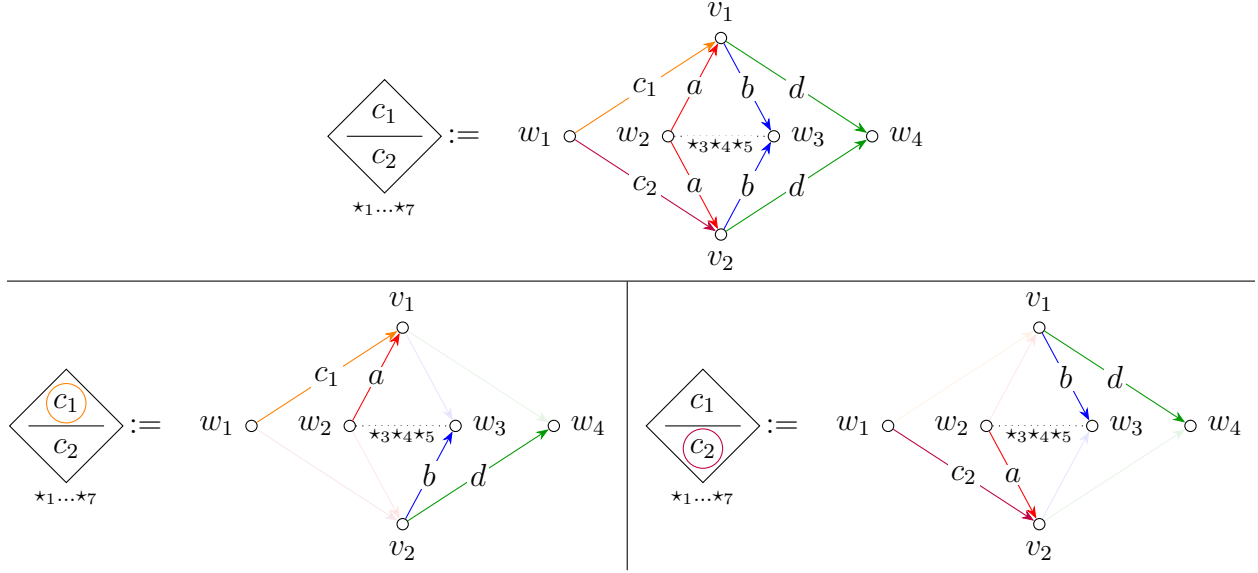

\subsection{Vertex gadget}

Our next gadget is a $\{v_1,v_2\}$-gadget for two vertices $v_1$ and $v_2$. An important ingredient in this gadget is the simple color gadget. Within the proof, there are two different constructions; the first is a simpler construction which requires two consecutive edges in the same direction, and the second assumes a longer sequence of alternating edges. 

\begin{lemma}[Vertex-Gadget Lemma]\label{lem:vtxgadget}
Let $\frac{1}{n} \ll \ep\ll\nu\ll \alpha \leq 1$.
Let $D$ be an $(n,\alpha,\ep)$-edge-color-pseudorandom and $2\nu$-color-gadget-pseudorandom digraph, and let $\star\in\arcset^{48}$. Then for every $V'\subseteq V(D)$ and $\C'\subseteq \C(D)$ such that $|V'|,|\C'|\geq (1-\nu)n$, and for every $v_1,v_2\in V(D)$, there exists an $(\{v_1,v_2\},\star)$-absorber with internal vertices in $V'$ and internal colors in $\C'$. 
\end{lemma}

\begin{proof}
We have two constructions depending on $\star$. In the first construction, we suppose that there exists $i \in \{2,\dots, 7\}$ such that $\star_i=\star_{i+1}$. We find the first half of the gadget pictured in \Cref{fig:ordvertexgadget} using \Cref{lem:shortpaths} (this half is rainbow), and then we use some color gadgets to account for the different colors used in the $v_1$- and $v_2$-absorbing paths.

Let $\C_{i-1}=\C^{-\star_{i-1}}(v_1)\cap\C'$, $\C_{i}=\C^{\star_i}(v_1)\cap\C'$, $\C_{i+1}=\C^{-\star_{i+1}}(v_2)\cap\C'$, and $\C_{i+2}=\C^{\star_{i+2}}(v_2)\cap\C'$. By \Cref{def:pseudo} (\ref{pseudo:colordiversity}), each of these has size at least $\alpha n/2 -\nu n\geq \alpha n/4$. Let $\C_j'\sbeq \C_j$ be chosen to be disjoint and each of size $|\C_j'|=\alpha n/16$ for $j\in\{i-1,i,i+1,i+2\}$.     

Now let $S_{i-1}=N^{-\star_{i-1}}_{\C_{i-1}'}(v_1)\cap V'$, $S_{i}=N^{\star_{i}}_{\C_{i}'}(v_1)\cap V'$, $S_{i+1}=N^{-\star_{i+1}}_{\C_{i+1}'}(v_2)\cap V'$, and $S_{i+2}=N^{\star_{i+2}}_{\C_{i+2}'}(v_2)\cap V'$, which each have size at least $\alpha n/16-\nu n\geq \alpha n/32\gg8\ep n$. Let $S_j'\sbeq S_j$ be chosen to be disjoint and each of size $|S_j'|=2\ep n$ for $j\in\{i-1,i,i+1,i+2\}$. For $j\in [18] \setminus \{i-1,i,i+1,i+2\}$, let $S_j'\sbeq V'$ be chosen to be each of size $2\ep n$ and so that the sets $S_1',\ldots, S_{18}'$ are pairwise disjoint.

We apply \Cref{lem:shortpaths} to find a rainbow path $w_1 \dots w_{18}$ through $S_1', \ldots S_{18}'$ with orientation $(\star_1,\star_2,\dots,\star_{i-1},\star_i=\star_{i+1},\star_{i+2},\ldots, \star_{18})$, that is, $(\star_1,\ldots, \star_{18} )$ with index $i+1$ removed, where $w_j\in S_j'$ for all $j \in [18]$, and where the color set $\D$ of the path is a subset of $\C' \setminus ( \C_{i-1}' \cup \C_i' \cup \C_{i+1}' \cup \C_{i+2}' )$, which has size at least $(1-\nu-\alpha/4)n \geq 18\ep n$. Let $c_1=\C((w_{i-1}w_i)^{\star_{i-1}})$, $c_2=\C((w_{i-1}v_1)^{\star_{i-1}}),c_3=\C((v_1w_i)^{\star_{i}})$, $d_1=\C((w_{i+1}w_{i+2})^{\star_{i+2}})$, $d_2=\C((w_{i+1}v_2)^{\star_{i+1}})$, and $d_3=\C((v_2w_{i+2})^{\star_{i+2}})$, as pictured in \Cref{fig:ordvertexgadget}. 

For $j \in [3]$, define $\star'_j = (\star_{21+10(j-1)+1},\ldots,\star_{21+10(j-1)+7})$. Since $D$ is $2\nu$-color-gadget-pseudorandom, by \Cref{def:cgpseudorandom}, we may successively find vertex-disjoint and color-disjoint $(\{c_j,d_j\},\star'_j)$-absorbers $(A_j,s_j,t_j)$ with internal vertices in $V'\setminus \{v_1, v_2, w_1,\ldots, w_{18}\}$ and internal colors in $\C'\setminus (\D\cup\{c_2,c_3,d_2,d_3\})$.

Lastly, apply \Cref{lem:connecting} to successively find vertex-disjoint and color-disjoint rainbow paths of length three (with internal vertices in $V'$ and disjoint from the $A_j$s, and colors in $\C'$ and disjoint from the $A_j$s) from $w_{18}$ to $s_1$ with orientation $(\star_{19},\star_{20},\star_{21})$, from $t_1$ to $s_2$ with orientation $(\star_{29},\star_{30},\star_{31})$, and from $t_2$ to $s_3$ with orientation $(\star_{39},\star_{40},\star_{41})$. See \Cref{fig:ordvertexgadget}. 

The $v_1$-absorbing path arises from taking the path $w_1\ldots w_{i-1}v_1w_iw_{i+1}\ldots w_{18}$ for the first 18 edges and following the $c_1$-absorbing path of $A_1$, the $d_2$-absorbing path of $A_2$, and the $d_3$-absorbing path of $A_3$. Similarly, the $v_2$-absorbing path arises from taking the path $w_1\ldots w_{i+1}v_{2}w_{i+2}\ldots w_{18}$ for the first 18 edges and following the $d_1$-absorbing path of $A_1$, the $c_2$-absorbing path of $A_2$, and the $c_{3}$-absorbing path of $A_3$.  

\begin{figure}
\begin{center}

\begin{tikzpicture}[
    >=Stealth,
    vertex/.style={circle,draw,inner sep=1.5pt},
    lbl/.style={draw=none,fill=none,inner sep=1pt},
    edgelabel/.style={midway, fill=white, inner sep=1pt, text=black}
]

\tikzset{
    pics/vertexgadget/.style args={#1/#2/#3}{
      code={
        \draw (0,0) -- (.375,.65) -- (1.125,.65) -- (1.5,0) -- (1.125,-.65) -- (.375,-.65) -- cycle;
        \draw (0.25,0) -- (1.25,0);
        \node[draw,draw opacity=0,circle,inner sep=1pt] at (.75,0.3) {#1};
        \node[draw,,draw opacity =0,circle,inner sep=1pt] at (.75,-0.3) {#2};
        \node at (.85,-.85) {#3};

        \coordinate (-left)  at (0,0);
        \coordinate (-right) at (2,0);
      }
    },
    pics/colorgadget/.style args={#1/#2/#3/#4/#5}{
      code={
        \draw (0,0) -- (.75,.75) -- (1.5,0) -- (.75,-.75) -- cycle;
        \draw (0.25,0) -- (1.25,0);
        \node[draw=#4,circle,inner sep=1pt] at (.75,0.3) {#1};
        \node[draw=#5,circle,inner sep=.8pt] at (.75,-0.3) {#2};
        \node at (.75,-.95) {#3};

        \coordinate (-left)  at (0,0);
        \coordinate (-right) at (1.5,0);
      }
    }
}

\node[vertex,label=below:$w_1$] (w1) at (2,0) {};
\node[vertex,label=below:$w_{i-1}$] (wi-1) at (3.5,0) {};
\node[vertex,label=below:$w_i$] (wi) at (5,0) {};
\node[vertex,label=below:$w_{i+1}$] (wi+1) at (6,0) {};
\node[vertex,label=below:$w_{i+2}$] (wi+2) at (7.5,0) {};
\node[vertex,label=above:$v_1$] (v1) at (4.25,1.3) {};
\node[vertex,label=above:$v_2$] (v2) at (6.75,1.3) {};

\draw[dotted] (w1) -- node[lbl,above=2pt,font=\scriptsize] {$\star_{1}\ldots\star_{i-2}$} (wi-1);

\draw[<-, red] (wi-1) -- node[edgelabel] {$c_1$} (wi);
\draw[<-, yellow!95!black] (wi-1) -- node[edgelabel] {$c_2$} (v1);
\draw[->, blue] (v1) -- node[edgelabel] {$c_3$} (wi);

\draw[->] (wi) -- (wi+1);

\draw[<-, orange] (wi+1) -- node[edgelabel] {$d_1$} (wi+2);
\draw[->, green!60!black] (wi+1) -- node[edgelabel] {$d_2$} (v2);
\draw[<-, purple] (v2) -- node[edgelabel] {$d_3$} (wi+2);

\pic (cd1) at (8.75,0) {colorgadget={$c_1$}/{$d_1$}/{\scriptsize$\star_{22}\ldots\star_{28}$}/{red,draw opacity =0}/{orange,draw opacity =0}};
\pic (cd2) at (11.45,0) {colorgadget={$c_2$}/{$d_2$}/{\scriptsize$\star_{32}\ldots\star_{38}$}/{yellow!95!black,draw opacity =0}/{green!60!black,draw opacity =0}};
\pic (cd3) at (14.1,0) {colorgadget={$c_3$}/{$d_3$}/{\scriptsize$\star_{42}\ldots\star_{48}$}/{blue,draw opacity =0}/{purple,draw opacity =0}};

\draw[dotted] (wi+2) -- node[lbl,above=2pt,font=\scriptsize] {$\star_{i+3}\ldots\star_{21}$} (cd1-left);
\draw[dotted] (cd1-right) -- node[lbl,above=2pt,font=\scriptsize] {$\star_{29}\ldots\star_{31}$} (cd2-left);
\draw[dotted] (cd2-right) -- node[lbl,above=2pt,font=\scriptsize] {$\star_{39}\ldots\star_{41}$} (cd3-left);

\begin{scope}[xshift=.8cm]
    \pic at (-1.2,0) {vertexgadget={$v_1$}/{$v_2$}/{\scriptsize$\star_1,\ldots,\star_{48}$}};
    \node at (.6,0) {$:=$};
\end{scope}

\end{tikzpicture}

\hrule

\begin{tikzpicture}[
    >=Stealth,
    vertex/.style={circle,draw,inner sep=1.5pt},
    lbl/.style={draw=none,fill=none,inner sep=1pt},
    edgelabel/.style={midway, fill=white, inner sep=1pt, text=black}
]

\tikzset{
    pics/vertexgadget/.style args={#1/#2/#3}{
      code={
        \draw (0,0) -- (.375,.65) -- (1.125,.65) -- (1.5,0) -- (1.125,-.65) -- (.375,-.65) -- cycle;
        \draw (0.25,0) -- (1.25,0);
        \node[draw,circle,inner sep=1pt] at (.75,0.3) {#1};
        \node at (.75,-0.3) {#2};
        \node at (.85,-.85) {#3};

        \coordinate (-left)  at (0,0);
        \coordinate (-right) at (2,0);
      }
    },
    pics/colorgadget/.style args={#1/#2/#3/#4/#5}{
      code={
        \draw (0,0) -- (.75,.75) -- (1.5,0) -- (.75,-.75) -- cycle;
        \draw (0.25,0) -- (1.25,0);
        \node[draw=#4,circle,inner sep=1pt] at (.75,0.3) {#1};
        \node[draw=#5,circle,inner sep=.8pt] at (.75,-0.3) {#2};
        \node at (.75,-.95) {#3};
        \coordinate (-left)  at (0,0);
        \coordinate (-right) at (1.5,0);
      }
    }
}

\node[vertex,label=below:$w_1$] (w1) at (2,0) {};
\node[vertex,label=below:$w_{i-1}$] (wi-1) at (3.5,0) {};
\node[vertex,label=below:$w_i$] (wi) at (5,0) {};
\node[vertex,label=below:$w_{i+1}$] (wi+1) at (6,0) {};
\node[vertex,label=below:$w_{i+2}$] (wi+2) at (7.5,0) {};
\node[vertex,label=above:$v_1$] (v1) at (4.25,1.3) {};
\node[vertex] (v2) at (6.75,1.3) {};

\draw[dotted] (w1) -- node[lbl,above=2pt,font=\scriptsize] {$\star_{1}\ldots\star_{i-2}$} (wi-1);

\draw[<-, red, opacity=0.1] (wi-1) -- (wi);

\draw[<-, yellow!95!black] (wi-1) -- node[edgelabel] {$c_2$} (v1);
\draw[->, blue] (v1) -- node[edgelabel] {$c_3$} (wi);

\draw[->] (wi) -- (wi+1);

\draw[<-, orange] (wi+1) -- node[edgelabel] {$d_1$} (wi+2);

\draw[->, green!60!black, opacity=0.1] (wi+1) -- (v2);
\draw[<-, purple, opacity=0.1] (v2) -- (wi+2);

\pic (cd1) at (8.75,0) {colorgadget={$c_1$}/{$d_1$}/{\scriptsize$\star_{22}\ldots\star_{28}$}/{red}/{orange, draw opacity=0}};
\pic (cd2) at (11.45,0) {colorgadget={$c_2$}/{$d_2$}/{\scriptsize$\star_{32}\ldots\star_{38}$}/{yellow!95!black, draw opacity=0}/{green}};
\pic (cd3) at (14.1,0) {colorgadget={$c_3$}/{$d_3$}/{\scriptsize$\star_{42}\ldots\star_{48}$}/{blue, draw opacity=0}/{purple}};

\draw[dotted] (wi+2) -- node[lbl,above=2pt,font=\scriptsize] {$\star_{i+3}\ldots\star_{21}$} (cd1-left);
\draw[dotted] (cd1-right) -- node[lbl,above=2pt,font=\scriptsize] {$\star_{29}\ldots\star_{31}$} (cd2-left);
\draw[dotted] (cd2-right) -- node[lbl,above=2pt,font=\scriptsize] {$\star_{39}\ldots\star_{41}$} (cd3-left);

\begin{scope}[xshift=.8cm]
    \pic at (-1.2,0) {vertexgadget={$v_1$}/{$v_2$}/{\scriptsize$\star_1,\ldots,\star_{48}$}};
    \node at (.6,0) {$:=$};
\end{scope}

\end{tikzpicture}

\hrule

\begin{tikzpicture}[
    >=Stealth,
    vertex/.style={circle,draw,inner sep=1.5pt},
    lbl/.style={draw=none,fill=none,inner sep=1pt},
    edgelabel/.style={midway, fill=white, inner sep=1pt, text=black}
]

\tikzset{
    pics/vertexgadget/.style args={#1/#2/#3}{
      code={
        \draw (0,0) -- (.375,.65) -- (1.125,.65) -- (1.5,0) -- (1.125,-.65) -- (.375,-.65) -- cycle;
        \draw (0.25,0) -- (1.25,0);
        \node at (.75,0.3) {#1};
        \node[draw,circle,inner sep=1pt] at (.75,-0.3) {#2};
        \node at (.85,-.85) {#3};

        \coordinate (-left)  at (0,0);
        \coordinate (-right) at (2,0);
      }
    },
    pics/colorgadget/.style args={#1/#2/#3/#4/#5}{
      code={
        \draw (0,0) -- (.75,.75) -- (1.5,0) -- (.75,-.75) -- cycle;
        \draw (0.25,0) -- (1.25,0);
        \node[draw=#4,circle,inner sep=1pt] at (.75,0.3) {#1};
        \node[draw=#5,circle,inner sep=.8pt] at (.75,-0.3) {#2};
        \node at (.75,-.95) {#3};
        \coordinate (-left)  at (0,0);
        \coordinate (-right) at (1.5,0);
      }
    }
}

\node[vertex,label=below:$w_1$] (w1) at (2,0) {};
\node[vertex,label=below:$w_{i-1}$] (wi-1) at (3.5,0) {};
\node[vertex,label=below:$w_i$] (wi) at (5,0) {};
\node[vertex,label=below:$w_{i+1}$] (wi+1) at (6,0) {};
\node[vertex,label=below:$w_{i+2}$] (wi+2) at (7.5,0) {};
\node[vertex] (v1) at (4.25,1.3) {};
\node[vertex,label=above:$v_2$] (v2) at (6.75,1.3) {};

\draw[dotted] (w1) -- node[lbl,above=2pt,font=\scriptsize] {$\star_{1}\ldots\star_{i-2}$} (wi-1);

\draw[<-, red] (wi-1) -- node[edgelabel] {$c_1$} (wi);

\draw[<-, yellow!95!black, opacity=0.1] (wi-1) -- (v1);
\draw[->, blue, opacity=0.1] (v1) -- (wi);

\draw[->] (wi) -- (wi+1);

\draw[<-, orange, opacity=0.1] (wi+1) -- (wi+2);

\draw[->, green!60!black] (wi+1) -- node[edgelabel] {$d_2$} (v2);
\draw[<-, purple] (v2) -- node[edgelabel] {$d_3$} (wi+2);

\pic (cd1) at (8.75,0) {colorgadget={$c_1$}/{$d_1$}/{\scriptsize$\star_{22}\ldots\star_{28}$}/{red, draw opacity=0}/{orange}};
\pic (cd2) at (11.45,0) {colorgadget={$c_2$}/{$d_2$}/{\scriptsize$\star_{32}\ldots\star_{38}$}/{yellow!95!black}/{green, draw opacity=0}};
\pic (cd3) at (14.1,0) {colorgadget={$c_3$}/{$d_3$}/{\scriptsize$\star_{42}\ldots\star_{48}$}/{blue}/{purple, draw opacity=0}};

\draw[dotted] (wi+2) -- node[lbl,above=2pt,font=\scriptsize] {$\star_{i+3}\ldots\star_{21}$} (cd1-left);
\draw[dotted] (cd1-right) -- node[lbl,above=2pt,font=\scriptsize] {$\star_{29}\ldots\star_{31}$} (cd2-left);
\draw[dotted] (cd2-right) -- node[lbl,above=2pt,font=\scriptsize] {$\star_{39}\ldots\star_{41}$} (cd3-left);

\begin{scope}[xshift=.8cm]
    \pic at (-1.2,0) {vertexgadget={$v_1$}/{$v_2$}/{\scriptsize$\star_1,\ldots,\star_{48}$}};
    \node at (.6,0) {$:=$};
\end{scope}

\end{tikzpicture}

\end{center}
\caption{$(\{v_1,v_2\},\star)$-absorber with $\star\in\arcset^{48}$, $(\star_{i-1},~\star_i,~\star_{i+1},~\star_{i+2})=(\leftarrow~,~\rightarrow~,~\rightarrow~,~\leftarrow)$ for some $3\leq i\leq 7$, and $v_1$-absorbing and $v_2$-absorbing paths shown. We use the short-hand notation for this gadget given on the left side of the figure in later constructions.}
\label{fig:ordvertexgadget}
\end{figure}
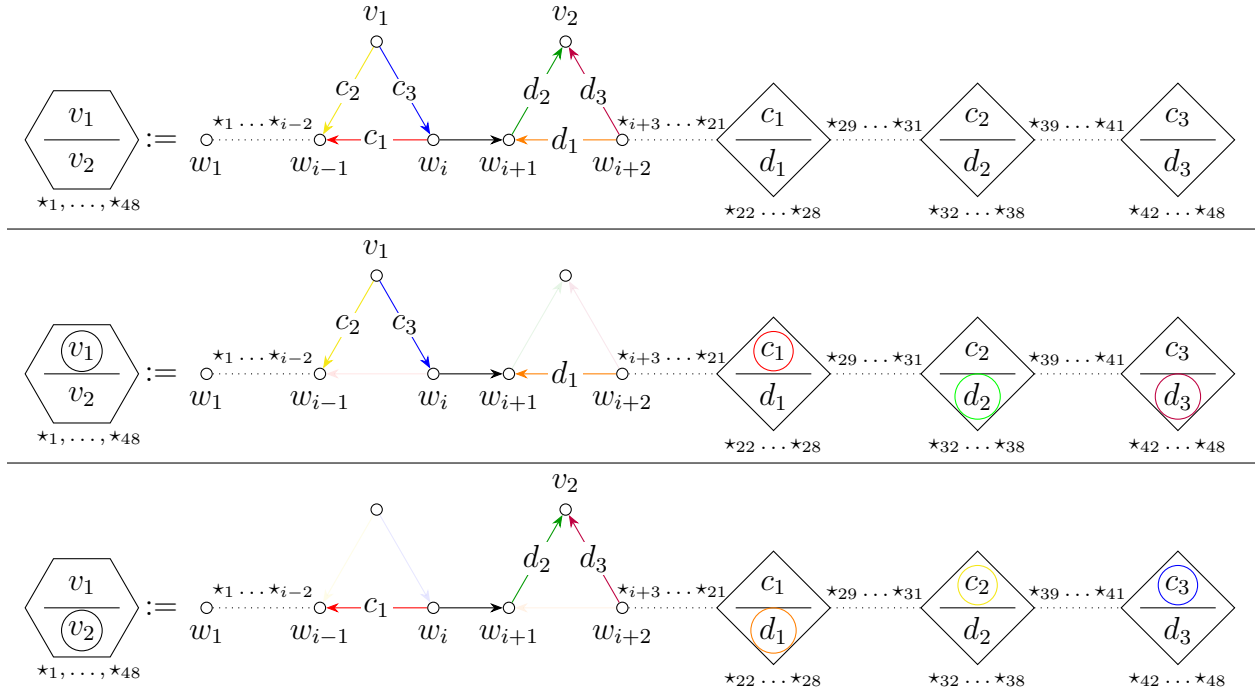

For the second construction, we have that $(\star_2,\ldots, \star_8)$ is alternating. As with the first construction, we find the first half of the gadget pictured in \Cref{fig:altvertexgadget} using two applications of \Cref{lem:shortpaths} (this half is rainbow), and then we use some color gadgets to account for the different colors in the $v_1$- and $v_2$-absorbing paths.

Let $\C_1=\C^{\star_3}(v_1)\cap \C'$ and $\C_3=\C^{\star_{5}}(v_2)\cap\C'$. By \Cref{def:pseudo} (\ref{pseudo:colordiversity}) we have $|\C_j|\geq \alpha n/2-\nu n\geq \alpha n/4$ for $j\in\{1,3\}$. Let $\C_j'\sbeq \C_j$ be disjoint such that $|\C_j|=\alpha n/8$ for $j\in\{1,3\}$. Let $S_1= N^{\star_{3}}_{\C_1'}(v_1)\cap V'$ and $S_3=N^{\star_{5}}_{\C_3'}(v_2)\cap V'$, so that $|S_j|\geq \alpha n/8-\nu n\geq \alpha n/16\gg 4\ep n$. Thus we can choose $S_j'\sbeq S_j$ to be disjoint and each of size $|S_j'|=2\ep n$ for $j\in\{1,3\}$; we also choose $S_2'\sbeq V'$ to be of size $|S_2'|=2\ep n$ so that the sets $S_1',S_2',S_3'$ are pairwise disjoint. 

Apply \Cref{lem:shortpaths} to find a rainbow path $w_1 w_2 w_3$ through $S_1',S_2',S_3'$ with orientation $(\star_4,\star_5)$, with $w_j\in S_j'$ for $j \in [3]$, and colors in $\C' \setminus (\C_1 \cup \C_3)$. Let $\D\sbeq \C'$ be the color set of the rainbow path path $v_1 w_1 w_2 w_3 v_2$.

Let $\D_2=\C^{-\star_2}(v_1)\cap(\C'\setminus \D)$, $\D_3=\C^{\star_6}(w_3)\cap (\C'\setminus \D)$, $\D_4=\C^{-\star_4}(v_2)\cap (\C'\setminus \D)$, and $\D_5=\C^{\star_8}(w_1)\cap (\C'\setminus\D)$. By \Cref{def:pseudo} (\ref{pseudo:colordiversity}) we have $|\D_j|\geq \alpha n/2-\nu n-4\geq \alpha n/4$. Let $\D_j'\subseteq \D_j$ be chosen to be disjoint and each of size $|\D_j'|=\alpha n/16$ for $j\in\{2,3,4,5\}$. 

Now let $T_{2}=N^{-\star_2}_{\D_2'}(v_1)\cap V'$, $T_3=N^{\star_6}_{\D_3'}(w_3)\cap V'$, $T_4=N^{-\star_4}_{\D_4'}(v_2)\cap V'$, $T_5=N^{\star_8}_{\D_5'}(w_1)\cap V'$, which each have size at least $\alpha n/16-\nu n\geq\alpha n/32\gg 8\ep n$. Thus we can choose $T_j'\subseteq T_j$ to be disjoint and each of size $|T_j'|=2\ep n$ for $j\in\{2,3,4,5\}$; we also choose $T_1'\subseteq V'$ to be of size $|T_1'|=2\ep n$ so that the sets $T_1',\ldots,T_5'$ are pairwise disjoint.

We apply \Cref{lem:shortpaths} to find a rainbow path $u_1\dots u_5$ through $T_1',\ldots, T_5'$ with orientation $(\star_1,\star_2,\star_3,\star_8)$ with $u_j\in T_j'$ for $j \in [5]$ and color set $\D' \subseteq \C' \setminus (\D \cup \D_2 \cup \D_3 \cup \D_4 \cup \D_5)$. Let $c_1 = \C((u_1v_1)^{\star_2})$, $c_2 = \C((v_1 w_1)^{\star_3})$, $c_3 = \C((w_3u_3)^{\star_6})$, $c_4 = \C((u_4u_5)^{\star_8})$, $d_1 = \C((u_2u_3)^{\star_2})$, $d_2 = \C((u_4v_2)^{\star_4})$, $d_3 = \C((v_2w_3)^{\star_5})$, $d_4 = \C((w_1u_5)^{\star_8})$. 

For $j \in [4]$, define $\star_j' = (\star_{10j+2,\ldots, \star_{10j+8}})$. Since $D$ is $2\nu$-color-gadget-pseudorandom, we may successively find vertex-disjoint and color-disjoint $(\{c_j,d_j\},\star_j')$-absorbers $(A_j,s_j,t_j)$ for $j\in [4]$ with internal vertices in $V' \setminus \{v_1, v_2, w_1, w_2, w_3, u_1, \dots, u_5\}$ and internal colors in $\C' \setminus (\D \cup \D' \cup \{c_1, c_2, c_3, c_4, d_1, d_2, d_3, d_4\})$.

Lastly, apply \Cref{lem:connecting} to successively find vertex-disjoint and color-disjoint rainbow paths of length three (with internal vertices in $V'$ and disjoint from the $A_j$s, and colors in $\C'$ and disjoint from the $A_j$s) from $u_5$ to $s_1$ with orientation $(\star_9,\star_{10},\star_{11})$ and from $t_j$ to $s_{j+1}$ with orientation $(\star_{10(i-1)+9},\ldots, \star_{10(i-1)+11})$ for $j\in[3]$, vertex- and color-disjoint from everything else found so far. See \Cref{fig:altvertexgadget}.

The $v_1$-absorbing path arises from following the path $u_1u_2v_1w_1w_2w_3u_3u_4u_5$ for the first eight edges, and activating the $d_i$-absorbing paths for each $A_i$. Similarly, the $v_2$-absorbing path arises from following the path $u_1u_2u_3u_4v_2w_3w_2w_1u_5$ for the first eight edges, and activating the $c_i$-absorbing paths for each $A_i$.
\end{proof}

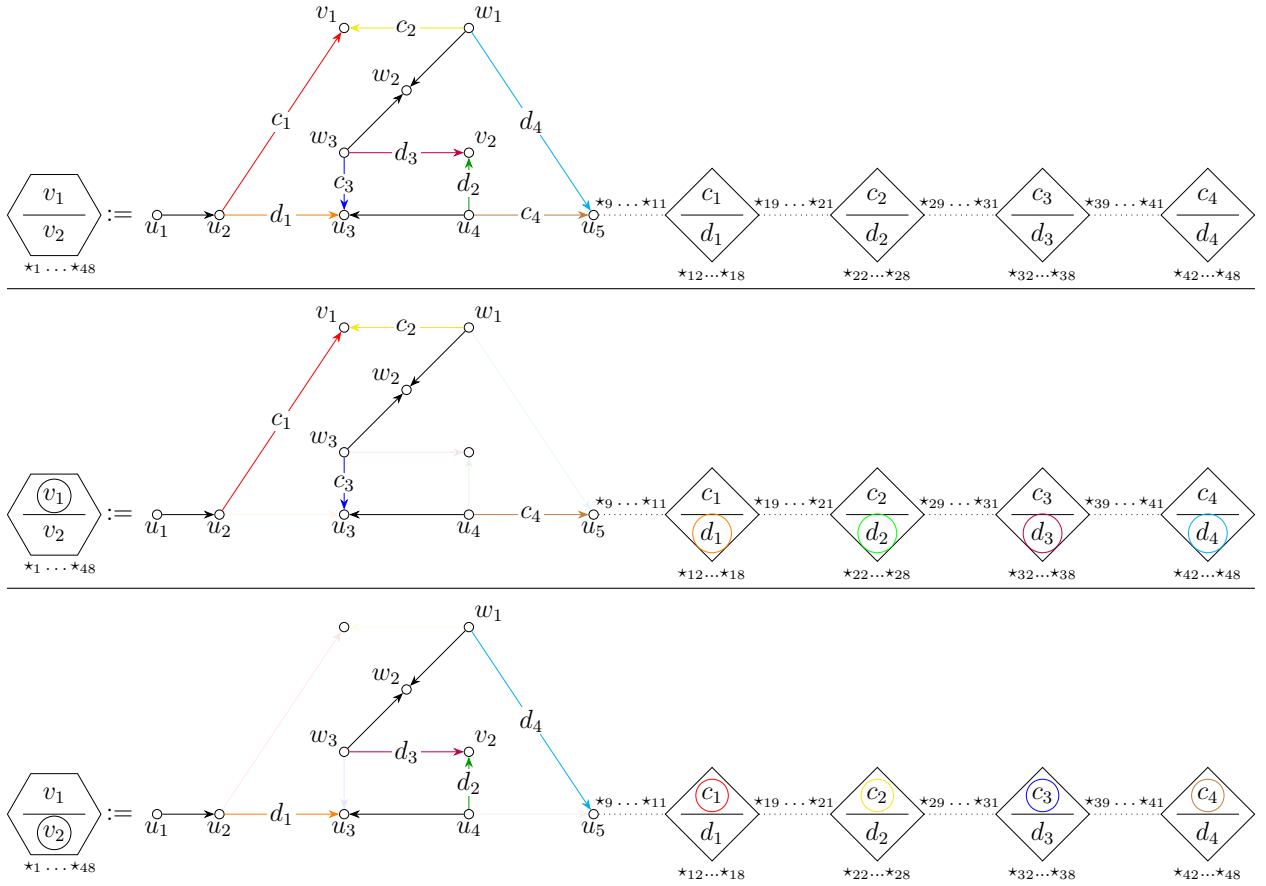
\begin{figure}
\begin{center}
\begin{adjustbox}{width=\textwidth}
\begin{tikzpicture}[
    >=Stealth,
    vertex/.style={circle,draw,inner sep=1.5pt},
    lbl/.style={draw=none,fill=none,inner sep=1pt},
    edgelabel/.style={midway, fill=white, inner sep=1pt, text=black}
]

\tikzset{
    pics/vertexgadget/.style args={#1/#2/#3}{
      code={
        \draw (0,0) -- (.375,.65) -- (1.125,.65) -- (1.5,0) -- (1.125,-.65) -- (.375,-.65) -- cycle;
        \draw (0.25,0) -- (1.25,0);
        \node at (.75,0.3) {#1};
        \node at (.75,-0.3) {#2};
        \node at (.85,-.85) {#3};
      }
    },
    pics/colorgadget/.style args={#1/#2/#3/#4/#5}{
      code={
        \draw (0,0) -- (.75,.75) -- (1.5,0) -- (.75,-.75) -- cycle;
        \draw (0.25,0) -- (1.25,0);
        \node[draw=#4,circle,inner sep=1pt] at (.75,0.3) {#1};
        \node[draw=#5,circle,inner sep=1pt] at (.75,-0.3) {#2};
        \node at (.75,-.95) {#3};
        \coordinate (-left)  at (0,0);
        \coordinate (-right) at (1.5,0);
      }
    }
}

\pic at (-1.4,0) {vertexgadget={$v_1$}/{$v_2$}/{\scriptsize$\star_1\ldots\star_{48}$}};
\node at (0.4,0) {$:=$};

\node[vertex] (u1) at (1,0) {};
\node[lbl, below = 2pt] at (u1) {$u_1$};

\node[vertex] (u2) at (2,0) {};
\node[lbl, below = 2pt] at (u2) {$u_2$};

\node[vertex] (u3) at (4,0) {};
\node[lbl, below = 2pt] at (u3) {$u_3$};

\node[vertex] (u4) at (6,0) {};
\node[lbl, below = 2pt] at (u4) {$u_4$};

\node[vertex] (u5) at (8,0) {};
\node[lbl, below = 2pt] at (u5) {$u_5$};

\node[vertex] (v1) at (4,3) {};
\node[lbl, above left = 2pt] at (v1) {$v_1$};

\node[vertex] (w3) at (4,1) {};
\node[lbl, above left = 2pt] at (w3) {$w_3$};

\node[vertex] (v2) at (6,1) {};
\node[lbl, above right = 2pt] at (v2) {$v_2$};

\node[vertex] (w2) at (5,2) {};
\node[lbl, above left = 2pt] at (w2) {$w_2$};

\node[vertex] (w1) at (6,3) {};
\node[lbl, above right = 2pt] at (w1) {$w_1$};

\draw[->] (u1) -- (u2);

\draw[->, orange] (u2) -- node[edgelabel] {$d_1$} (u3);
\draw[->] (u4) -- (u3);
\draw[->, brown] (u4) -- node[edgelabel] {$c_4$} (u5);

\draw[->, red] (u2) -- node[edgelabel] {$c_1$} (v1);
\draw[->, green!60!black] (u4) -- node[edgelabel] {$d_2$} (v2);

\draw[->, blue] (w3) -- node[edgelabel] {$c_3$} (u3);
\draw[->, purple] (w3) -- node[edgelabel] {$d_3$} (v2);
\draw[->] (w3) -- (w2);

\draw[->] (w1) -- (w2);
\draw[->, cyan] (w1) -- node[edgelabel] {$d_4$} (u5);
\draw[->, yellow!95!black] (w1) -- node[edgelabel] {$c_2$} (v1);

\pic (cd1) at (9.15,0) {colorgadget={$c_1$}/{$d_1$}/{$\scriptstyle \star_{12}\ldots\star_{18}$}/{red,draw opacity =0}/{orange,draw opacity =0}};
\pic (cd2) at (11.8,0) {colorgadget={$c_2$}/{$d_2$}/{$\scriptstyle \star_{22}\ldots\star_{28}$}/{yellow!95!black,draw opacity =0}/{green,draw opacity =0}};
\pic (cd3) at (14.45,0) {colorgadget={$c_3$}/{$d_3$}/{$\scriptstyle \star_{32}\ldots\star_{38}$}/{blue,draw opacity =0}/{purple,draw opacity =0}};
\pic (cd4) at (17.1,0) {colorgadget={$c_4$}/{$d_4$}/{$\scriptstyle \star_{42}\ldots\star_{48}$}/{brown,draw opacity =0}/{cyan,draw opacity =0}};

\draw[dotted] (u5) -- node[lbl,above=2pt,font=\scriptsize] {$\star_{9}\ldots\star_{11}$} (cd1-left);
\draw[dotted] (cd1-right) -- node[lbl,above=2pt,font=\scriptsize] {$\star_{19}\ldots\star_{21}$} (cd2-left);
\draw[dotted] (cd2-right) -- node[lbl,above=2pt,font=\scriptsize] {$\star_{29}\ldots\star_{31}$} (cd3-left);
\draw[dotted] (cd3-right) -- node[lbl,above=2pt,font=\scriptsize] {$\star_{39}\ldots\star_{41}$} (cd4-left);

\end{tikzpicture}
\end{adjustbox}

\hrule
\vspace{2mm}

\begin{adjustbox}{width=\textwidth}
    \begin{tikzpicture}[
    >=Stealth,
    vertex/.style={circle,draw,inner sep=1.5pt},
    lbl/.style={draw=none,fill=none,inner sep=1pt},
    edgelabel/.style={midway, fill=white, inner sep=1pt, text=black}
]

\tikzset{
    pics/vertexgadget/.style args={#1/#2/#3}{
      code={
        \draw (0,0) -- (.375,.65) -- (1.125,.65) -- (1.5,0) -- (1.125,-.65) -- (.375,-.65) -- cycle;
        \draw (0.25,0) -- (1.25,0);
        \node[draw, circle,inner sep=1pt] at (.75,0.3) {#1};
        \node at (.75,-0.3) {#2};
        \node at (.85,-.85) {#3};
      }
    },
    pics/colorgadget/.style args={#1/#2/#3/#4/#5}{
      code={
        \draw (0,0) -- (.75,.75) -- (1.5,0) -- (.75,-.75) -- cycle;
        \draw (0.25,0) -- (1.25,0);
        \node[draw=#4,circle,inner sep=1pt] at (.75,0.3) {#1};
        \node[draw=#5,circle,inner sep=1pt] at (.75,-0.3) {#2};
        \node at (.75,-.95) {#3};
        \coordinate (-left)  at (0,0);
        \coordinate (-right) at (1.5,0);
      }
    }
}

\pic at (-1.4,0) {vertexgadget={$v_1$}/{$v_2$}/{\scriptsize$\star_1\ldots\star_{48}$}};
\node at (0.4,0) {$:=$};

\node[vertex] (u1) at (1,0) {};
\node[lbl, below = 2pt] at (u1) {$u_1$};
\node[vertex] (u2) at (2,0) {};
\node[lbl, below = 2pt] at (u2) {$u_2$};
\node[vertex] (u3) at (4,0) {};
\node[lbl, below = 2pt] at (u3) {$u_3$};
\node[vertex] (u4) at (6,0) {};
\node[lbl, below = 2pt] at (u4) {$u_4$};
\node[vertex] (u5) at (8,0) {};
\node[lbl, below = 2pt] at (u5) {$u_5$};

\node[vertex] (v1) at (4,3) {};
\node[lbl, above left = 2pt] at (v1) {$v_1$};
\node[vertex] (w3) at (4,1) {};
\node[lbl, above left = 2pt] at (w3) {$w_3$};
\node[vertex] (v2) at (6,1) {};
\node[lbl, above right = 2pt] at (v2) {};
\node[vertex] (w2) at (5,2) {};
\node[lbl, above left = 2pt] at (w2) {$w_2$};
\node[vertex] (w1) at (6,3) {};
\node[lbl, above right = 2pt] at (w1) {$w_1$};

\draw[->] (u1) -- (u2);

\draw[->, orange, opacity=0.1] (u2) -- (u3);
\draw[->] (u4) -- (u3);
\draw[->, brown] (u4) -- node[edgelabel] {$c_4$} (u5);

\draw[->, red] (u2) -- node[edgelabel] {$c_1$} (v1);
\draw[->, green!60!black, opacity=0.1] (u4) -- (v2);

\draw[->, blue] (w3) -- node[edgelabel] {$c_3$} (u3);
\draw[->, purple, opacity=0.1] (w3) -- (v2);
\draw[->] (w3) -- (w2);

\draw[->] (w1) -- (w2);
\draw[->, cyan, opacity=0.1] (w1) -- (u5);
\draw[->, yellow!95!black] (w1) -- node[edgelabel] {$c_2$} (v1);

\pic (cd1) at (9.15,0) {colorgadget={$c_1$}/{$d_1$}/{$\scriptstyle \star_{12}\ldots\star_{18}$}/{red, draw opacity=0}/{orange}};
\pic (cd2) at (11.8,0) {colorgadget={$c_2$}/{$d_2$}/{$\scriptstyle \star_{22}\ldots\star_{28}$}/{yellow!95!black, draw opacity=0}/{green}};
\pic (cd3) at (14.45,0) {colorgadget={$c_3$}/{$d_3$}/{$\scriptstyle \star_{32}\ldots\star_{38}$}/{blue, draw opacity=0}/{purple}};
\pic (cd4) at (17.1,0) {colorgadget={$c_4$}/{$d_4$}/{$\scriptstyle \star_{42}\ldots\star_{48}$}/{brown, draw opacity=0}/{cyan}};

\draw[dotted] (u5) -- node[lbl,above=2pt,font=\scriptsize] {$\star_{9}\ldots\star_{11}$} (cd1-left);
\draw[dotted] (cd1-right) -- node[lbl,above=2pt,font=\scriptsize] {$\star_{19}\ldots\star_{21}$} (cd2-left);
\draw[dotted] (cd2-right) -- node[lbl,above=2pt,font=\scriptsize] {$\star_{29}\ldots\star_{31}$} (cd3-left);
\draw[dotted] (cd3-right) -- node[lbl,above=2pt,font=\scriptsize] {$\star_{39}\ldots\star_{41}$} (cd4-left);

\end{tikzpicture}
\end{adjustbox}

\hrule
\vspace{2mm}

\begin{adjustbox}{width=\textwidth}
    \begin{tikzpicture}[
    >=Stealth,
    vertex/.style={circle,draw,inner sep=1.5pt},
    lbl/.style={draw=none,fill=none,inner sep=1pt},
    edgelabel/.style={midway, fill=white, inner sep=1pt, text=black}
]

\tikzset{
    pics/vertexgadget/.style args={#1/#2/#3}{
      code={
        \draw (0,0) -- (.375,.65) -- (1.125,.65) -- (1.5,0) -- (1.125,-.65) -- (.375,-.65) -- cycle;
        \draw (0.25,0) -- (1.25,0);
        \node at (.75,0.3) {#1};
        \node[draw,circle,inner sep=1pt] at (.75,-0.3) {#2};
        \node at (.85,-.85) {#3};
      }
    },
    pics/colorgadget/.style args={#1/#2/#3/#4/#5}{
      code={
        \draw (0,0) -- (.75,.75) -- (1.5,0) -- (.75,-.75) -- cycle;
        \draw (0.25,0) -- (1.25,0);
        \node[draw=#4,circle,inner sep=1pt] at (.75,0.3) {#1};
        \node[draw=#5,circle,inner sep=1pt] at (.75,-0.3) {#2};
        \node at (.75,-.95) {#3};
        \coordinate (-left)  at (0,0);
        \coordinate (-right) at (1.5,0);
      }
    }
}

\pic at (-1.4,0) {vertexgadget={$v_1$}/{$v_2$}/{\scriptsize$\star_1\ldots\star_{48}$}};
\node at (0.4,0) {$:=$};

\node[vertex] (u1) at (1,0) {};
\node[lbl, below = 2pt] at (u1) {$u_1$};
\node[vertex] (u2) at (2,0) {};
\node[lbl, below = 2pt] at (u2) {$u_2$};
\node[vertex] (u3) at (4,0) {};
\node[lbl, below = 2pt] at (u3) {$u_3$};
\node[vertex] (u4) at (6,0) {};
\node[lbl, below = 2pt] at (u4) {$u_4$};
\node[vertex] (u5) at (8,0) {};
\node[lbl, below = 2pt] at (u5) {$u_5$};

\node[vertex] (v1) at (4,3) {};
\node[lbl, above left = 2pt] at (v1) {};
\node[vertex] (w3) at (4,1) {};
\node[lbl, above left = 2pt] at (w3) {$w_3$};
\node[vertex] (v2) at (6,1) {};
\node[lbl, above right = 2pt] at (v2) {$v_2$};
\node[vertex] (w2) at (5,2) {};
\node[lbl, above left = 2pt] at (w2) {$w_2$};
\node[vertex] (w1) at (6,3) {};
\node[lbl, above right = 2pt] at (w1) {$w_1$};

\draw[->] (u1) -- (u2);

\draw[->, orange] (u2) -- node[edgelabel] {$d_1$} (u3);
\draw[->] (u4) -- (u3);
\draw[->, brown, opacity=0.1] (u4) -- (u5);

\draw[->, red, opacity=0.1] (u2) -- (v1);
\draw[->, green!60!black] (u4) -- node[edgelabel] {$d_2$} (v2);

\draw[->, blue, opacity=0.1] (w3) -- (u3);
\draw[->, purple] (w3) -- node[edgelabel] {$d_3$} (v2);
\draw[->] (w3) -- (w2);

\draw[->] (w1) -- (w2);
\draw[->, cyan] (w1) -- node[edgelabel] {$d_4$} (u5);
\draw[->, yellow!95!black, opacity=0.1] (w1) -- (v1);

\pic (cd1) at (9.15,0) {colorgadget={$c_1$}/{$d_1$}/{$\scriptstyle \star_{12}\ldots\star_{18}$}/{red}/{orange, draw opacity=0}};
\pic (cd2) at (11.8,0) {colorgadget={$c_2$}/{$d_2$}/{$\scriptstyle \star_{22}\ldots\star_{28}$}/{yellow!95!black}/{green, draw opacity=0}};
\pic (cd3) at (14.45,0) {colorgadget={$c_3$}/{$d_3$}/{$\scriptstyle \star_{32}\ldots\star_{38}$}/{blue}/{purple, draw opacity=0}};
\pic (cd4) at (17.1,0) {colorgadget={$c_4$}/{$d_4$}/{$\scriptstyle \star_{42}\ldots\star_{48}$}/{brown}/{cyan, draw opacity=0}};

\draw[dotted] (u5) -- node[lbl,above=2pt,font=\scriptsize] {$\star_{9}\ldots\star_{11}$} (cd1-left);
\draw[dotted] (cd1-right) -- node[lbl,above=2pt,font=\scriptsize] {$\star_{19}\ldots\star_{21}$} (cd2-left);
\draw[dotted] (cd2-right) -- node[lbl,above=2pt,font=\scriptsize] {$\star_{29}\ldots\star_{31}$} (cd3-left);
\draw[dotted] (cd3-right) -- node[lbl,above=2pt,font=\scriptsize] {$\star_{39}\ldots\star_{41}$} (cd4-left);

\end{tikzpicture}
\end{adjustbox}
\end{center}
\caption{$(\{v_1,v_2\},\star)$-absorber with $\star\in\arcset^{48}$, $\star_1=\rightarrow$, $\star_2=\rightarrow$, $(\star_2,\ldots, \star_8)$ alternating, and $v_1$-absorbing and $v_2$-absorbing paths shown. We use the short-hand notation for this gadget given on the left side of the figure in later constructions.}
\label{fig:altvertexgadget}
\end{figure}

\section{Local absorbers}\label{sec:local_abs}

Recall \Cref{def:absorber} of an $\mathcal{S}$-absorber. 
What we refer to as a `local absorber' is a $\mathcal{S}$-absorber where $|\mathcal{S}| = 40$, that is, the local absorber absorbs exactly one out of a given set of $40$ options. We first construct a `special' local absorber, where each $S_i \in \mathcal{S}$ contains two colors and one vertex, and only exists for particular choices of $S_i$. We then use our gadgets of \Cref{sec:gadgets} to leverage the special local absorbers into local absorbers which exist for any set of $40$ colors or vertices. The proofs for local vertex absorbers and local color absorbers are roughly the same, the difference being whether we use vertex or color gadgets in some steps.

\subsection{Special local absorber}

\begin{lemma}[Special local absorber]\label{lem:specialabsorber}
Let $\frac{1}{n} \ll \ep \ll \nu \ll \alpha \leq 1$.
Let $D$ be an $(n,\alpha,\ep)$-edge-color-pseudorandom digraph, and let $\star \in \arcset^2$.
Then for every $V'\sbeq V$ and $\C'\sbeq \C$ with $|V'|,|\C'|\geq (1-\nu)n$, there exist distinct vertices $v_1,\ldots, v_{40} \in V'$ and distinct colors $c_1,\ldots, c_{40}, d_1,\ldots, d_{40} \in \C'$ such that there exists a $(\{\{v_i,c_i,d_i\}\}_{i=1}^{40},\star)$-absorber in $D$ with internal vertices in $V'$ and internal colors in $\C'$.
\end{lemma}

\begin{proof}
Let $K\in V(D)^{42}=(u_1,u_2,v_1,\ldots, v_{40})$. We say $K$ is a $K_{2,40}$ with orientation $\star$ if $(u_1v_i)^{\star_1}, (v_iu_2)^{\star_2} \in E(D)$ for every $i\in [40]$. We let $D[K]$ denote the subdigraph of $D$ consisting of only these vertices and edges. We say that $K$ is rainbow if $D[K]$ is rainbow. Our goal is to find a rainbow $K_{2,40}$ with orientation $\star$ and colors in $\C'$, since this is a $(\{\{v_i,c_i,d_i\}\}_{i=1}^{40},\star)$-absorber (with no internal colors and internal vertices $u_1$ and $u_2$) where $c_i$ is the color of $u_1v_i$ and $d_i$ is the color if $v_i u_2$; see \Cref{fig:specialasborber}.

Note that a $K_{2,40}$ with orientation $\star$ is a blow-up of a consistently oriented path.
We apply \Cref{lem:supersat} to $D[V']$ to obtain a set $\K \subseteq V(D)^{42}$ of size $|\K| \geq \theta n^{42}$ such that each $K \in \K$ is a $K_{2,40}$ with orientation $\star$ and verrtices in $V'$. Using the pseudorandomness properties, we show that the number of non-rainbow $K \in \K$ and the number of $K \in \K$ which use a color not in $\C'$ is much smaller than $\theta n^{42}$, so there must exist $K \in \K$ which is rainbow with colors in $\C'$, as desired.

We first count the number of $K \in \K$ using a color not in $\C'$. There are $\nu n$ choices for a color not in $\C'$. Given that color, there are at most $10n$ edges of $D$ of that color by \Cref{def:pseudo} (\ref{pseudo:avgedges}). There are $80$ edges of $D[K]$ for which a given edge of that color can play the role. Finally, given that edge, there are at most $n^{40}$ ways to choose the other $40$ vertices of $K$. Thus in total, there are at most $\nu n \cdot 10 n \cdot 80 \cdot n^{40} = 800 \nu n^{42} \ll \theta n^{42}$ many $K \in \K$ using a color not in $\C'$, since $\nu \ll \alpha$ and $\theta$ depends on $\alpha$ via \Cref{lem:supersat}.

Second, we count the number of $K \in \K$ which are not rainbow: there must be two edges $e,f \in E(D[K])$ with the same color. We break into two cases:
\begin{enumerate}
\item Assume $e=(xy)^{*_e}$ and $f=(zw)^{*_f}$ are not incident. There are at most $n^{40}$ choices for $\{u_1,u_2,v_1,\ldots, v_{40}\}\setminus \{z,w\}$. Once $x$ and $y$ are fixed, there are at most $10n$ choices for $\{z,w\}$ such that $\C(f)=\C(e)$ by \Cref{def:pseudo} (\ref{pseudo:avgedges}). Since there are at most $80^2$ choices for $e$ and $f$, the total number of such $K\in\K$ is at most $80^2\cdot 10n^{41} \ll \theta n^{42}$.
\item Assume $e=(xw)^{*_e}$ and $f=(wy)^{*_f}$ are incident. There are at most $n^{41}$ choices for $\{u_1,u_2,v_1,\ldots, v_{40}\}\setminus \{y\}$. Once $x$ and $w$ are fixed, there are at most $n/\log n$ choices for $y\in N_{\C(e)}^{*_f}(w)$ so that $\C(f)=\C(e)$ by \Cref{def:pseudo} (\ref{pseudo:maxdeg}) on vertex $w$. Since there are at most $80^2$ choices for $e$ and $f$, the total number of such $K\in\K$ is at most $80^2\cdot n^{42}/\log n \ll \theta n^{42}$. 
\end{enumerate}
\begin{figure}
    \begin{center}
  \begin{tikzpicture}[>=Stealth, line width=0.5pt]

\tikzset{
  pics/k240/.style n args={4}{
    code={
      \node[circle, draw, inner sep=1.3pt] (-L) at (-2,0) {};
      \node[circle, draw, inner sep=1.3pt] (-T) at (0,2) {};
      \node[circle, draw, inner sep=1.3pt] (-R) at (2,0) {};
      \node[circle, draw, inner sep=1.3pt] (-B) at (0,-2) {};
      \node[circle, draw, inner sep=1.3pt] (-M) at (0,.9) {};

      \draw[->, red] (-L) -- (-T);
      \draw[->, yellow!95!black] (-L) -- (-M);
      \draw[->, blue] (-L) -- (-B);

      \draw[->, orange] (-R) -- (-T);
      \draw[->, green!60!black] (-R) -- (-M);
      \draw[->, purple] (-R) -- (-B);

      \node[fill=white, inner sep=1pt] at (-1.00,1.12) {$c_1$};
      \node[fill=white, inner sep=1pt] at (-1.00,0.53) {$c_2$};
      \node[fill=white, inner sep=1pt] at (-1.05,-1.05) {$c_{40}$};

      \node[fill=white, inner sep=1pt] at (1.00,1.12) {$d_1$};
      \node[fill=white, inner sep=1pt] at (1.00,0.53) {$d_2$};
      \node[fill=white, inner sep=1pt] at (1.05,-1.05) {$d_{40}$};

      \node at (0,-0.25) {$\vdots$};

      \node[left=4pt]  at (-L) {#1};
      \node[above=4pt] at (-T) {#2};
      \node[right=4pt] at (-R) {#3};
      \node[below=4pt] at (-B) {#4};
      \node[above=3pt] at (-M) {$v_2$};
    }
  }
}

\pic at (0,0) {k240={$u_1$}{$v_1$}{$u_2$}{$v_{40}$}};

\end{tikzpicture}
    \end{center}
    \caption{A $(\{\{v_i,c_i,d_i\}\}_{i=1}^{40},(\rightarrow,\leftarrow))$-absorber. The $\{v_i,c_i,d_i\}$-absorbing path is simply $u_1 v_i u_2$.}
    \label{fig:specialasborber}
\end{figure}
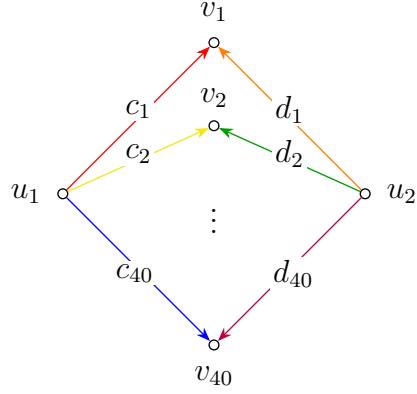
\end{proof}

\subsection{Local vertex absorbers}

\begin{lemma}[Local vertex absorber]\label{lem:localvtxabs}
Let $\frac{1}{n} \ll \ep \ll \nu \ll \alpha \leq 1$.
Let $D$ be an $(n,\alpha,\ep)$-edge-color-pseudorandom and $3\nu$-color-gadget pseudorandom digraph, and let $\star \in \arcset^{2822}$.
Then for every $V' \subseteq V$ and $\C' \subseteq \C$ with $|V'|, |\C'| \geq (1-\nu)n$ and for every $v_1,\dots,v_{40} \in V'$, there exists a $(\{v_1,\dots,v_{40}\},\star)$-absorber in $D$ with internal vertices in $V'$ and internal colors in $\C'$.
\end{lemma}

\begin{proof}
See \Cref{fig:localvtxabs}. Define the following subpatterns of $\star$:
\begin{itemize}
\item $\star_i^1 = (\star_{51(i-1)+1},\ldots,\star_{51(i-1)+48})$ for $i \in [40]$,
\item $\star_i^2 = (\star_{2046+10(i-1)},\ldots, \star_{2052+10(i-1)})$ for $i \in [39]$,
\item $\star_i^3 = (\star_{2436+10(i-1)},\ldots, \star_{2442+10(i-1)})$ for $i \in [39]$,
\item $\star_i^4 = (\star_{51i-2},\star_{51i-1},\star_{51i})$ for $i \in [40]$,
\item $\star_i^5 = (\star_{10i+2033},\star_{10i+2034},\star_{10i+2035})$ for $i \in [39]$,
\item $\star_i^6 = (\star_{10i+2434},\star_{10i+2435},\star_{10i+2436})$ for $i \in [39]$.
\end{itemize}

First we apply \Cref{lem:specialabsorber} to find a $(\{\{c_i,w_i,d_i\}\}_{i=1}^{40},(\star_{2041},\star_{2042}))$-absorber $(A_0,s_0,t_0)$. By \Cref{lem:vtxgadget}, we may successively find vertex-disjoint and color-disjoint $(\{v_i,w_i\},\star_i^1)$-absorbers $(A_i,s_i,t_i)$ for $i\in [40]$ with vertices and colors in $V'$ and $\C'$, and internally vertex and internally color disjoint from $A_0$. Similarly, by \Cref{def:cgpseudorandom}, we may successively find $(\{c_i,c_{i+1}\}, \star_i^2)$-absorbers $(G_i,x_i,y_i)$ and $(\{d_i,d_{i+1}\}, \star_i^3)$-absorbers $(G_i',x_i',y_i')$ for $i\in [39]$ which are vertex-disjoint and color-disjoint from each other, internally vertex- and color-disjoint from the previously found absorbers, and with vertices in $V'$ and colors in $\C'$.

Lastly, we apply \Cref{lem:connecting} to successively find vertex-disjoint and color-disjoint paths of length three (internally vertex-disjoint and color-disjoint from all absorbers found thus far, with vertices and colors in $V'$ and $\C'$) from $t_i$ to $s_{i+1}$ oriented by $\star_i^4$ for $i\in [39]$, from $t_{40}$ to $s_0$ oriented by $\star_{40}^4$, from $t_0$ to $x_1$ oriented by $\star_{1}^5$, from $y_i$ to $x_{i+1}$ oriented by $\star_{i+1}^5$ for $i\in [38]$, from $y_{39}$ to $x_1'$ oriented by $\star_{1}^6$, from $y_i'$ to $x_{i+1}'$ oriented by $\star_{i+1}^6$ for $i \in [38]$.

The $v_i$-absorbing path arises from following the $v_i$-absorbing path in $A_i$, the $w_j$-absorbing path in $A_j$ for $j\neq i$, the $\{c_i,w_i,d_i\}$-absorbing path $A_0$, and lastly, the $c_j$ absorbing path for $j<i$ and the $c_{j+1}$-absorbing path for $j\geq i$ in $G_i$, and similarly the $d_j$ absorbing path for $j<i$ and the $d_{j+1}$-absorbing path for $j \geq i$ in $G_j'$.
\end{proof}

\begin{figure}
\begin{center}
\begin{adjustbox}{width=\textwidth}
   \begin{tikzpicture}[
    >=Stealth,
    vertex/.style={circle,draw,inner sep=1.5pt},
    lbl/.style={draw=none,fill=none,inner sep=1pt}
]

\tikzset{
  pics/vertexgadget/.style args={#1/#2/#3/#4/#5}{
    code={
      \draw (0,0) -- (.375,.65) -- (1.125,.65) -- (1.5,0) -- (1.125,-.65) -- (.375,-.65) -- cycle;
      \draw (0.25,0) -- (1.25,0);
      \node[draw,circle,inner sep=1pt,draw opacity=#4,text opacity=1] at (.75,0.3) {#1};
      \node[draw,circle,inner sep=1pt,draw opacity=#5,text opacity=1] at (.75,-0.3) {#2};
      \node at (.85,-.85) {#3};
      \coordinate (-left) at (0,0);
      \coordinate (-right) at (1.5,0);
    }
  },
  pics/colorgadget/.style args={#1/#2/#3/#4/#5}{
    code={
      \draw (0,0) -- (.75,.75) -- (1.5,0) -- (.75,-.75) -- cycle;
      \draw (0.25,0) -- (1.25,0);
      \node[draw=#4,circle,inner sep=.8pt] at (.75,0.3) {#1};
      \node[draw=#5,circle,inner sep=.8pt] at (.75,-0.3) {#2};
      \node at (.75,-.95) {#3};
      \coordinate (-left) at (0,0);
      \coordinate (-right) at (1.5,0);
    }
  },
  pics/k240/.style n args={2}{
    code={
      \node[circle, draw, inner sep=1.3pt] (-L) at (-1,0) {};
      \node[circle, draw, inner sep=1.3pt] (-T) at (0,1) {};
      \node[circle, draw, inner sep=1.3pt] (-R) at (1,0) {};
      \node[circle, draw, inner sep=1.3pt] (-B) at (0,-1) {};
      \node[circle, draw, inner sep=1.3pt] (-M) at (0,.45) {};

      \draw[red] (-L) -- (-T);
      \draw[yellow!95!black] (-L) -- (-M);
      \draw[blue] (-L) -- (-B);

      \draw[orange] (-R) -- (-T);
      \draw[green!60!black] (-R) -- (-M);
      \draw[purple] (-R) -- (-B);

      \node[fill=white, inner sep=1pt] at (-0.52,0.62) {$c_1$};
      \node[fill=white, inner sep=1pt] at (-0.55,-0.60) {$c_{40}$};

      \node[fill=white, inner sep=1pt] at (0.52,0.62) {$d_1$};
      \node[fill=white, inner sep=1pt] at (0.55,-0.60) {$d_{40}$};

      \node at (0,-0.15) {$\vdots$};

      \node[above=3pt] at (-T) {#1};
      \node[below=3pt] at (-B) {#2};
      \node[below=12pt] at (-B) {$\scriptstyle \star_{2041}\star_{2042}$};
    }
  }
}
\tikzset{
  pics/localabsorber/.style n args={5}{
    code={
      \draw (0,0) rectangle (1,2.5);

      \draw (0,2) -- (1,2);
      \draw (0,1.5) -- (1,1.5);
      \draw (0,1) -- (1,1);
      \draw (0,.5) -- (1,.5);

      \node at (0.5,2.25) {#1};
      \node at (0.5,1.75) {#2};
      \node at (0.5,1.35) {#3};
      \node at (0.5,.75) {#4};
      \node at (0.5,0.25) {#5};
    }
  }
}
\pic (c140) at (-1,-2.75) {localabsorber={$v_1$}{$v_2$}{$\vdots$}{$v_{39}$}{$v_{40}$}};
\node at (.5,-1.5) {$:=$};

\begin{scope}[xshift = 1 cm]
\pic (vw1) at (0,0) {vertexgadget={$v_1$}/{$w_1$}/{$\scriptstyle \star^1_1$}/{0}/{0}};
\pic (vw2) at (2.5,0) {vertexgadget={$v_2$}/{$w_2$}/{$\scriptstyle \star^1_2$}/{0}/{0}};
\pic (vw40) at (5.5,0) {vertexgadget={$v_{40}$}/{$w_{40}$}/{$\scriptstyle \star^1_{40}$}/{0}/{0}};

\pic (w140) at (9.5,0) {k240={$w_1$}{$w_{40}$}};

\draw[dotted] (vw1-right) -- node[lbl,above=2pt,font=\scriptsize] {$\star^4_1$} (vw2-left);
\draw[dotted] (vw2-right) -- node[lbl,above=2pt,font=\scriptsize] {$\star_{100}\ldots\star_{1989}$} (vw40-left);
\draw[dotted] (vw40-right) -- node[lbl,above=2pt,font=\scriptsize] {$\star^4_{40}$} (w140-L);
\draw[dotted] (w140-R) -- node[lbl,above=2pt,font=\scriptsize] {$\star^5_1$} (12,0);

\pic (c12) at (0.5,-3) {colorgadget={$c_1$}/{$c_2$}/{$\scriptstyle \star^2_1$}/{red,draw opacity =0}/{yellow!95!black,draw opacity =0}};
\pic (c23) at (3,-3) {colorgadget={$c_2$}/{$c_3$}/{$\scriptstyle \star^2_2$}/{yellow!95!black,draw opacity =0}/{cyan,draw opacity =0}};
\pic (c3940) at (6,-3) {colorgadget={$c_{39}$}/{$c_{40}$}/{$\scriptstyle \star^2_{39}$}/{brown,draw opacity =0}/{blue,draw opacity =0}};

\pic (d12) at (8.5,-3) {colorgadget={$d_1$}/{$d_2$}/{$\scriptstyle \star^3_1$}/{orange,draw opacity =0}/{green!60!black,draw opacity =0}};
\pic (d23) at (11,-3) {colorgadget={$d_2$}/{$d_3$}/{$\scriptstyle \star^3_2$}/{green!60!black,draw opacity =0}/{pink,draw opacity =0}};
\pic (d3940) at (14,-3) {colorgadget={$d_{39}$}/{$d_{40}$}/{$\scriptstyle \star^3_{39}$}/{teal,draw opacity =0}/{purple,draw opacity =0}};

\draw[dotted] (0,-3) -- (c12-left);
\draw[dotted] (c12-right) -- node[lbl,above=2pt,font=\scriptsize] {$\star^5_2$} (c23-left);
\draw[dotted] (c23-right) -- node[lbl,above=2pt,font=\scriptsize] {$\star_{2063}\ldots\star_{2435}$} (c3940-left);
\draw[dotted] (c3940-right) -- node[lbl,above=2pt,font=\scriptsize] {$\star^6_1$} (d12-left);
\draw[dotted] (d12-right) -- node[lbl,above=2pt,font=\scriptsize] {$\star^6_2$} (d23-left);
\draw[dotted] (d23-right) -- node[lbl,above=2pt,font=\scriptsize] {$\star_{2463}\ldots\star_{2815}$} (d3940-left);
\end{scope}
\end{tikzpicture}
\end{adjustbox}

\hrule

\begin{adjustbox}{width=\textwidth}
\begin{tikzpicture}[
    >=Stealth,
    vertex/.style={circle,draw,inner sep=1.5pt},
    lbl/.style={draw=none,fill=none,inner sep=1pt}
]

\tikzset{
  pics/vertexgadget/.style args={#1/#2/#3/#4/#5}{
    code={
      \draw (0,0) -- (.375,.65) -- (1.125,.65) -- (1.5,0) -- (1.125,-.65) -- (.375,-.65) -- cycle;
      \draw (0.25,0) -- (1.25,0);
      \node[draw,circle,inner sep=1pt,draw opacity=#4,text opacity=1] at (.75,0.3) {#1};
      \node[draw,circle,inner sep=1pt,draw opacity=#5,text opacity=1] at (.75,-0.3) {#2};
      \node at (.85,-.85) {#3};
      \coordinate (-left) at (0,0);
      \coordinate (-right) at (1.5,0);
    }
  },
  pics/colorgadget/.style args={#1/#2/#3/#4/#5}{
    code={
      \draw (0,0) -- (.75,.75) -- (1.5,0) -- (.75,-.75) -- cycle;
      \draw (0.25,0) -- (1.25,0);
      \node[draw=#4,circle,inner sep=.8pt] at (.75,0.3) {#1};
      \node[draw=#5,circle,inner sep=.8pt] at (.75,-0.3) {#2};
      \node at (.75,-.95) {#3};
      \coordinate (-left) at (0,0);
      \coordinate (-right) at (1.5,0);
    }
  },
  pics/k240/.style n args={2}{
    code={
      \node[circle, draw, inner sep=1.3pt] (-L) at (-1,0) {};
      \node[circle, draw, inner sep=1.3pt] (-T) at (0,1.25) {};
      \node[circle, draw, inner sep=1.3pt] (-R) at (1,0) {};
      \node[circle, draw, inner sep=1.3pt] (-B) at (0,-1) {};
      \node[circle, draw, inner sep=1.3pt] (-M) at (0,.45) {};

      \draw[red,draw opacity = .1] (-L) -- (-T);
      \draw[yellow!95!black] (-L) -- (-M);
      \draw[blue,draw opacity = .1] (-L) -- (-B);

      \draw[orange,draw opacity = .1] (-R) -- (-T);
      \draw[green!60!black] (-R) -- (-M);
      \draw[purple,draw opacity = .1] (-R) -- (-B);

      \node[fill=white, inner sep=1pt] at (-0.52,0.28) {$c_2$};

      \node[fill=white, inner sep=1pt] at (0.52,0.28) {$d_2$};

      \node at (0,-0.15) {$\vdots$};

      \node[above=3pt] at (-T) {};
      \node[below=3pt] at (-B) {};
      \node[above] at (-M) {$w_2$};
      
      \node[below=12pt] at (-B) {$\scriptstyle \star_{2041}\star_{2042}$};
    }
  }
}
\tikzset{
  pics/localabsorber/.style n args={5}{
    code={
      \draw (0,0) rectangle (1,2.5);

      \draw (0,2) -- (1,2);
      \draw (0,1.5) -- (1,1.5);
      \draw (0,1) -- (1,1);
      \draw (0,.5) -- (1,.5);

      \node at (0.5,2.25) {#1};
      \node[circle, draw, inner sep=1pt] at (0.5,1.75) {#2};
      \node at (0.5,1.35) {#3};
      \node at (0.5,.75) {#4};
      \node at (0.5,0.25) {#5};
    }
  }
}
\pic (c140) at (-1,-2.75) {localabsorber={$v_1$}{$v_2$}{$\vdots$}{$v_{39}$}{$v_{40}$}};
\node at (.5,-1.5) {$:=$};

\begin{scope}[xshift = 1 cm]
\pic (vw1) at (0,0)
  {vertexgadget={$v_1$}/{$w_1$}/{$\scriptstyle \star^1_1$}/{0}/{1}};

\pic (vw2) at (2.5,0)
  {vertexgadget={$v_2$}/{$w_2$}/{$\scriptstyle \star^1_2$}/{1}/{0}};

\pic (vw40) at (5.5,0)
  {vertexgadget={$v_{40}$}/{$w_{40}$}/{$\scriptstyle \star^1_{40}$}/{0}/{1}};

\pic (w140) at (9.5,0) {k240={$w_1$}{$w_{40}$}};

\draw[dotted] (vw1-right)
  -- node[lbl,above=2pt,font=\scriptsize] {$\star^4_1$}
  (vw2-left);

\draw[dotted] (vw2-right)
  -- node[lbl,above=2pt,font=\scriptsize] {$\star_{100}\ldots\star_{1989}$}
  (vw40-left);

\draw[dotted] (vw40-right)
  -- node[lbl,above=2pt,font=\scriptsize] {$\star^4_{40}$}
  (w140-L);

\draw[dotted] (w140-R)
  -- node[lbl,above=2pt,font=\scriptsize] {$\star^5_1$}
  (12,0);

\pic (c12) at (0.5,-3)
  {colorgadget={$c_1$}/{$c_2$}/{$\scriptstyle \star^2_1$}/{red}/{yellow!95!black,draw opacity=0}};

\pic (c23) at (3,-3)
  {colorgadget={$c_2$}/{$c_3$}/{$\scriptstyle \star^2_2$}/{yellow!95!black,draw opacity=0}/{cyan}};

\pic (c3940) at (6,-3)
  {colorgadget={$c_{39}$}/{$c_{40}$}/{$\scriptstyle \star^2_{39}$}/{brown,draw opacity=0}/{blue}};

\pic (d12) at (8.5,-3)
  {colorgadget={$d_1$}/{$d_2$}/{$\scriptstyle \star^3_1$}/{orange}/{green!60!black,draw opacity=0}};

\pic (d23) at (11,-3)
  {colorgadget={$d_2$}/{$d_3$}/{$\scriptstyle \star^3_2$}/{green!60!black,draw opacity=0}/{pink}};

\pic (d3940) at (14,-3)
  {colorgadget={$d_{39}$}/{$d_{40}$}/{$\scriptstyle \star^3_{39}$}/{black,draw opacity=0}/{purple}};

\draw[dotted] (0,-3) -- (c12-left);

\draw[dotted] (c12-right)
  -- node[lbl,above=2pt,font=\scriptsize] {$\star^5_2$}
  (c23-left);

\draw[dotted] (c23-right)
  -- node[lbl,above=2pt,font=\scriptsize] {$\star_{2063}\ldots\star_{2435}$}
  (c3940-left);

\draw[dotted] (c3940-right)
  -- node[lbl,above=2pt,font=\scriptsize] {$\star^6_1$}
  (d12-left);

\draw[dotted] (d12-right)
  -- node[lbl,above=2pt,font=\scriptsize] {$\star^6_2$}
  (d23-left);

\draw[dotted] (d23-right)
  -- node[lbl,above=2pt,font=\scriptsize] {$\star_{2463}\ldots\star_{2815}$}
  (d3940-left);
\end{scope}
\end{tikzpicture}

\end{adjustbox}
\end{center}

\caption{$(\{v_1,\ldots,v_{40}\},\star)$-absorber, with $v_2$-absorbing path shown. We use the short-hand notation for this gadget given on the left side of the figure in later constructions.}
\label{fig:localvtxabs}
\end{figure}
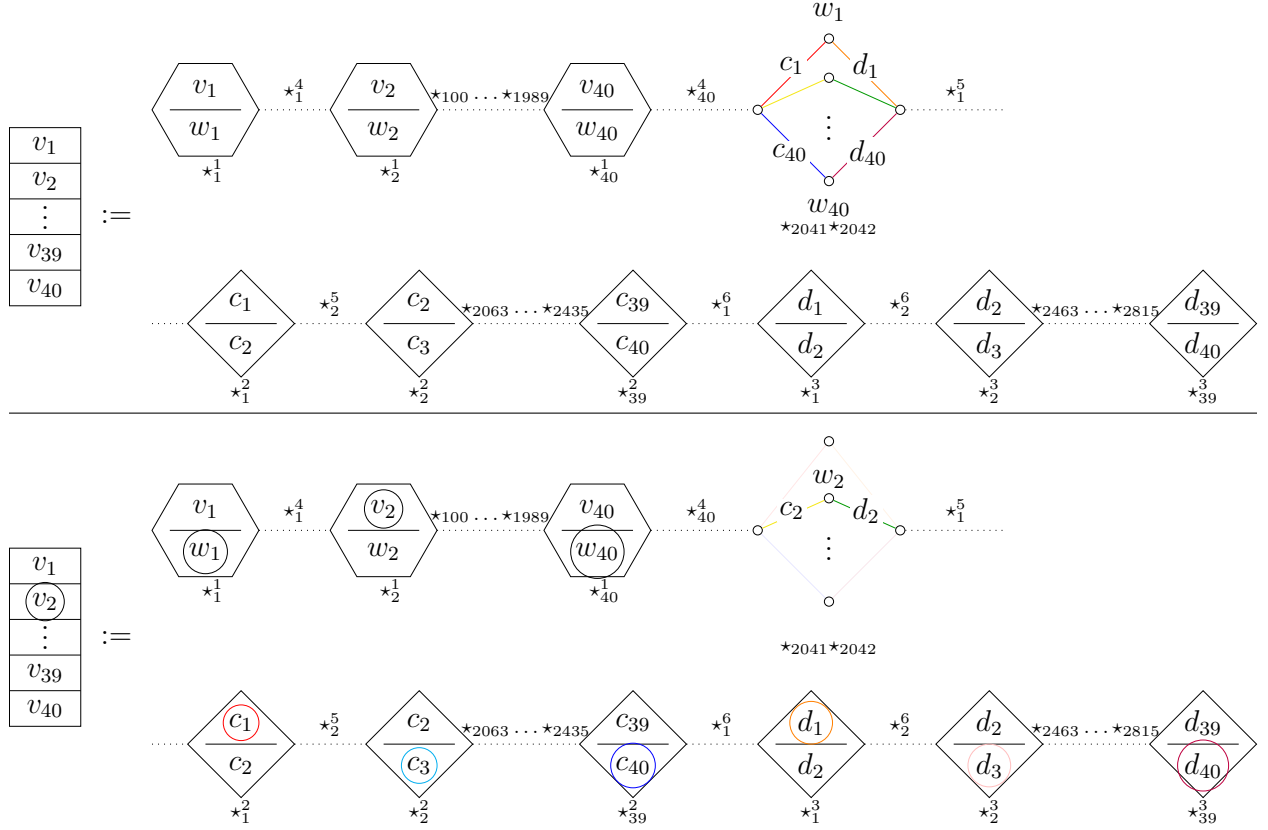

\subsection{Local color absorbers}

\begin{lemma}[Local color absorber]\label{lem:localcolorabs}
Let $\frac{1}{n} \ll \ep \ll \nu \ll \alpha \leq 1$.
Let $D$ be an $(n,\alpha,\ep$-edge-color-pseudorandom and $3\nu$-color-gadget pseudorandom digraph, and let $\star \in \arcset^{2781}$.
Then for every $V' \subseteq V$ and $\C' \subseteq \C$ with $|V'|, |\C'| \geq (1-\nu)n$ and for every $c_1,\dots,c_{40} \in \C'$, there exists an $(\{c_1,\dots,c_{40}\},\star)$-absorber in $D$ with internal vertices in $V'$ and internal colors in $\C'$.
\end{lemma}

\begin{proof}
See \Cref{fig:localcolabs}. Define the following subpatterns of $\star$:
\begin{itemize}
\item $\star_i^1 = (\star_{10(i-1)+1},\ldots,\star_{10(i-1)+7})$ for $i \in [40]$,
\item $\star_i^2 = (\star_{406+10(i-1)},\ldots, \star_{412+10(i-1)})$ for $i \in [39]$,
\item $\star_i^3 = (\star_{796+51(i-1)},\ldots, \star_{843+51(i-1)})$ for $i \in [39]$,
\item $\star_i^4 = (\star_{10i-2},\star_{10i-1},\star_{10i})$ for $i \in [39]$,
\item $\star_i^5 = (\star_{10i+393},\star_{10i+394},\star_{10i+395})$ for $i \in [39]$,
\item $\star_i^6 = (\star_{51(i-1)+793},\star_{51(i-1)+794},\star_{51(i-1)+795})$ for $i \in [39]$.
\end{itemize}

First we apply \Cref{lem:specialabsorber} to find a $(\{\{d_i,w_i,f_i\}\},(\star_{401},\star_{402}))$-absorber $(A_0,s_0,t_0)$. By \Cref{def:cgpseudorandom}, we may successively find vertex- and color-disjoint $(\{c_i,d_i\},\star_i^1)$-absorbers $(A_i,s_i,t_i)$ for $i\in [40]$ with vertices in $V'$ and colors in $\C'$, and internally vertex- and internally color-disjoint from $A_0$. Similarly, by \Cref{def:cgpseudorandom} and \Cref{lem:vtxgadget}, we may successively find $(\{f_i,f_{i+1}\}, \star_i^2)$-absorbers $(G_i,x_i,y_i)$ and $(\{w_i,w_{i+1}\}, \star_i^3)$-absorbers $(G_i',x_i',y_i')$ for $i\in [39]$ which are vertex- and color-disjoint from each other, internally vertex- and color-disjoint from the previously found absorbers, and with vertices in $V'$ and colors in $\C'$.

Lastly, we apply \Cref{lem:connecting} to successively find vertex-disjoint and color-disjoint paths of length three (internally vertex-disjoint and color-disjoint from all absorbers found thus far, with vertices and colors in $V'$ and $\C'$) from $t_i$ to $s_{i+1}$ oriented by $\star_i^4$ for $i\in\{1,\ldots,39\}$, from $t_{40}$ to $s_0$ oriented by $\star^4_{40}$, from $t_0$ to $x_1$ oriented by $\star^5_1$, from $y_i$ to $x_{i+1}$ oriented by $\star_{i+1}^5$ for $i\in [38]$, from $y_{39}$ to $x_1'$ oriented by $\star^6_1$, from $y_i'$ to $x_{i+1}'$ oriented by $\star_{i+1}^6$.

The $c_i$-absorbing path arises from following the $c_i$-absorbing path in $A_i$, and the $d_j$-absorbing path in $A_j$, for $j\neq i$. We then follow the $\{d_i,w_i,f_i\}$-absorbing path in $A_0$. Lastly, we follow the $f_j$-absorbing path for $j<i$ and the $f_{j+1}$-absorbing path for $j\geq i$ in $G_i$, and similarly the $w_j$ absorbing path for $j<i$ and the $w_{j+1}$-absorbing path for $j \geq i$ in $G_j'$.
\end{proof}

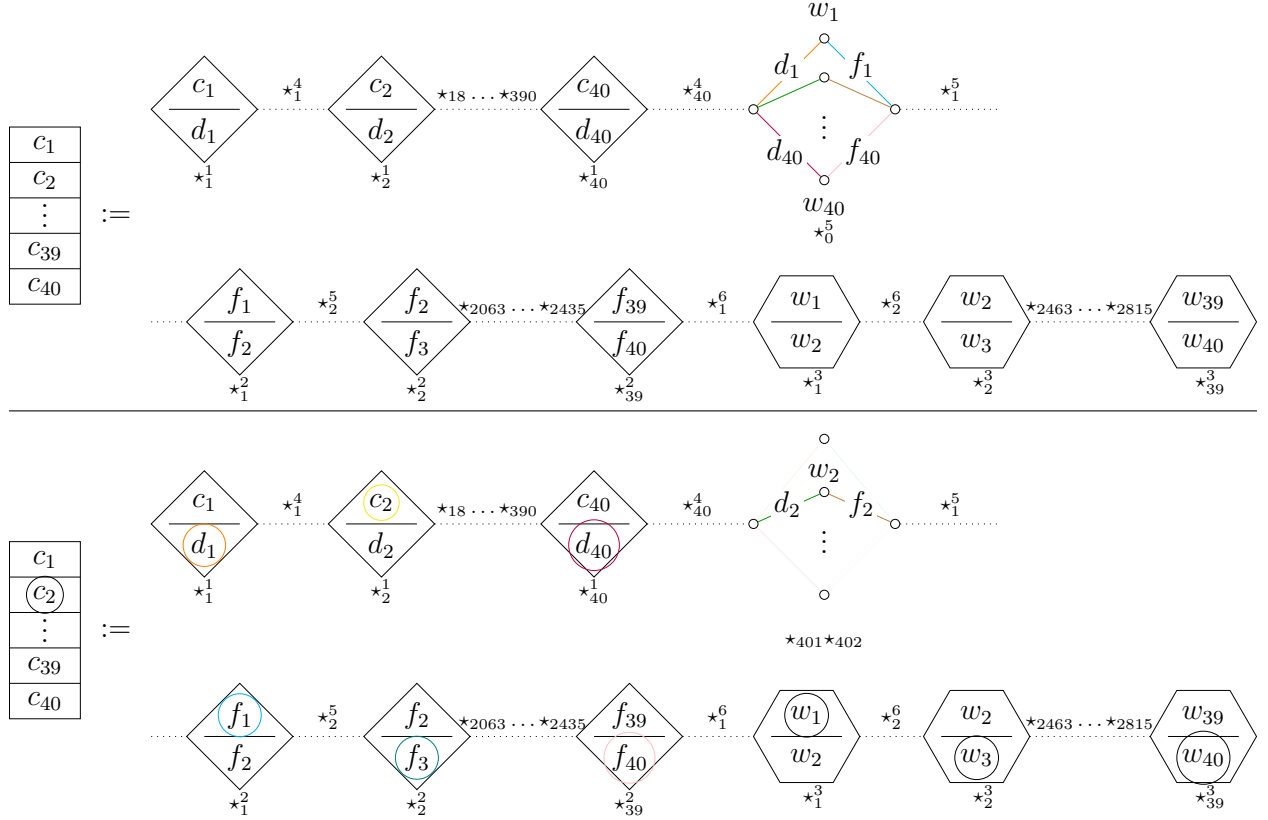
\begin{figure}
\begin{center}
\begin{adjustbox}{width=\textwidth}
\begin{tikzpicture}[
    >=Stealth,
    vertex/.style={circle,draw,inner sep=1.5pt},
    lbl/.style={draw=none,fill=none,inner sep=1pt}
]

\tikzset{
  pics/vertexgadget/.style n args={5}{
    code={
      \draw (0,0) -- (.375,.65) -- (1.125,.65) -- (1.5,0) -- (1.125,-.65) -- (.375,-.65) -- cycle;
      \draw (0.25,0) -- (1.25,0);
      \node[draw,circle,inner sep=1pt,draw opacity=#4,text opacity=1] at (.75,0.3) {#1};
      \node[draw,circle,inner sep=1pt,draw opacity=#5,text opacity=1] at (.75,-0.3) {#2};
      \node at (.85,-.85) {#3};

      \coordinate (-left)  at (0,0);
      \coordinate (-right) at (1.5,0);
    }
  }}
\tikzset{
  pics/colorgadget/.style n args={5}{
    code={
      \draw (0,0) -- (.75,.75) -- (1.5,0) -- (.75,-.75) -- cycle;
      \draw (0.25,0) -- (1.25,0);
      \node[#4,circle,inner sep=.8pt] at (.75,0.3) {#1};
      \node[#5,circle,inner sep=.8pt] at (.75,-0.3) {#2};
      \node at (.75,-.95) {#3};

      \coordinate (-left)  at (0,0);
      \coordinate (-right) at (1.5,0);
    }
  }}
\tikzset{
  pics/k240/.style n args={2}{
    code={
      \node[circle, draw, inner sep=1.3pt] (-L) at (-1,0) {};
      \node[circle, draw, inner sep=1.3pt] (-T) at (0,1) {};
      \node[circle, draw, inner sep=1.3pt] (-R) at (1,0) {};
      \node[circle, draw, inner sep=1.3pt] (-B) at (0,-1) {};
      \node[circle, draw, inner sep=1.3pt] (-M) at (0,.45) {};

      \draw[orange] (-L) -- (-T);
      \draw[green!60!black] (-L) -- (-M);
      \draw[purple] (-L) -- (-B);

      \draw[cyan] (-R) -- (-T);
      \draw[brown] (-R) -- (-M);
      \draw[pink] (-R) -- (-B);

      \node[fill=white, inner sep=1pt] at (-0.52,0.62) {$d_1$};
      \node[fill=white, inner sep=1pt] at (-0.55,-0.60) {$d_{40}$};

      \node[fill=white, inner sep=1pt] at (0.52,0.62) {$f_1$};
      \node[fill=white, inner sep=1pt] at (0.55,-0.60) {$f_{40}$};

      \node at (0,-0.15) {$\vdots$};

      \node[above=3pt] at (-T) {#1};
      \node[below=3pt] at (-B) {#2};      
      \node[below=12pt] at (-B) {$\scriptstyle \star^5_0$};
    }
  }
}
\tikzset{
  pics/localabsorber/.style n args={5}{
    code={
      \draw (0,0) rectangle (1,2.5);

      \draw (0,2) -- (1,2);
      \draw (0,1.5) -- (1,1.5);
      \draw (0,1) -- (1,1);
      \draw (0,.5) -- (1,.5);

      \node at (0.5,2.25) {#1};
      \node at (0.5,1.75) {#2};
      \node at (0.5,1.35) {#3};
      \node at (0.5,.75) {#4};
      \node at (0.5,0.25) {#5};
    }
  }
}
\pic (c140) at (-1,-2.75) {localabsorber={$c_1$}{$c_2$}{$\vdots$}{$c_{39}$}{$c_{40}$}};
\node at (.5,-1.5) {$:=$};

\begin{scope}[xshift = 1 cm]
\pic (vw1) at (0,0) {colorgadget={$c_1$}{$d_1$}{$\scriptstyle\star^1_1$}{draw=red,draw opacity =0}{draw=orange,draw opacity =0}};
\pic (vw2) at (2.5,0) {colorgadget={$c_2$}{$d_2$}{$\scriptstyle\star^1_2$}{draw=yellow!95!black,draw opacity =0}{draw=green!60!black,draw opacity =0}};
\pic (vw40) at (5.5,0) {colorgadget={$c_{40}$}{$d_{40}$}{$\scriptstyle \star^1_{40}$}{draw=blue,draw opacity =0}{draw=purple,draw opacity =0}};

\pic (w140) at (9.5,0) {k240={$w_1$}{$w_{40}$}}; 

\draw[dotted] (vw1-right) -- node[lbl,above=2pt,font=\scriptsize] {$\star^4_1$} (vw2-left);
\draw[dotted] (vw2-right) -- node[lbl,above=2pt,font=\scriptsize] {$\star_{18}\ldots\star_{390}$} (vw40-left);
\draw[dotted] (vw40-right) -- node[lbl,above=2pt,font=\scriptsize] {$\star^4_{40}$} (w140-L);
\draw[dotted] (w140-R) -- node[lbl,above=2pt,font=\scriptsize] {$\star^5_1$} (12,0);

\pic (e12) at (0.5,-3) {colorgadget={$f_1$}{$f_2$}{$\scriptstyle \star^2_1$}{draw=cyan,draw opacity =0}{draw=brown,draw opacity =0}};
\pic (e23) at (3,-3) {colorgadget={$f_2$}{$f_3$}{$\scriptstyle \star^2_2$}{draw=brown,draw opacity =0}{draw=teal,draw opacity =0}};
\pic (e3940) at (6,-3) {colorgadget={$f_{39}$}{$f_{40}$}{$\scriptstyle \star^2_{39}$}{draw=violet,draw opacity =0}{draw=pink,draw opacity =0}};

\pic (w12) at (8.5,-3) {vertexgadget={$w_1$}{$w_2$}{$\scriptstyle \star^3_1$}{0}{0}};
\pic (w23) at (10.9,-3) {vertexgadget={$w_2$}{$w_3$}{$\scriptstyle \star^3_2$}{0}{0}};
\pic (w3940) at (14.1,-3) {vertexgadget={$w_{39}$}{$w_{40}$}{$\scriptstyle \star^3_{39}$}{0}{0}};

\draw[dotted] (0,-3) -- (e12-left);
\draw[dotted] (e12-right) -- node[lbl,above=2pt,font=\scriptsize] {$\star^5_2$} (e23-left);
\draw[dotted] (e23-right) -- node[lbl,above=2pt,font=\scriptsize] {$\star_{2063}\ldots\star_{2435}$} (e3940-left);

\draw[dotted] (e3940-right) -- node[lbl,above=2pt,font=\scriptsize] {$\star^6_{1}$} (w12-left);
\draw[dotted] (w12-right) -- node[lbl,above=2pt,font=\scriptsize] {$\star^6_2$} (w23-left);
\draw[dotted] (w23-right) -- node[lbl,above=2pt,font=\scriptsize] {$\star_{2463}\ldots\star_{2815}$} (w3940-left);
\end{scope}
\end{tikzpicture}
\end{adjustbox}

\hrule

\begin{adjustbox}{width=\textwidth}
\begin{tikzpicture}[
    >=Stealth,
    vertex/.style={circle,draw,inner sep=1.5pt},
    lbl/.style={draw=none,fill=none,inner sep=1pt}
]

\tikzset{
  pics/vertexgadget/.style n args={5}{
    code={
      \draw (0,0) -- (.375,.65) -- (1.125,.65) -- (1.5,0) -- (1.125,-.65) -- (.375,-.65) -- cycle;
      \draw (0.25,0) -- (1.25,0);
      \node[draw,circle,inner sep=1pt,draw opacity=#4,text opacity=1] at (.75,0.3) {#1};
      \node[draw,circle,inner sep=1pt,draw opacity=#5,text opacity=1] at (.75,-0.3) {#2};
      \node at (.85,-.85) {#3};

      \coordinate (-left)  at (0,0);
      \coordinate (-right) at (1.5,0);
    }
  }}
\tikzset{
  pics/colorgadget/.style n args={5}{
    code={
      \draw (0,0) -- (.75,.75) -- (1.5,0) -- (.75,-.75) -- cycle;
      \draw (0.25,0) -- (1.25,0);
      \node[#4,circle,inner sep=.8pt] at (.75,0.3) {#1};
      \node[#5,circle,inner sep=.8pt] at (.75,-0.3) {#2};
      \node at (.75,-.95) {#3};

      \coordinate (-left)  at (0,0);
      \coordinate (-right) at (1.5,0);
    }
  }}
\tikzset{
  pics/k240/.style n args={2}{
    code={
      \node[circle, draw, inner sep=1.3pt] (-L) at (-1,0) {};
      \node[circle, draw, inner sep=1.3pt] (-T) at (0,1.2) {};
      \node[circle, draw, inner sep=1.3pt] (-R) at (1,0) {};
      \node[circle, draw, inner sep=1.3pt] (-B) at (0,-1) {};
      \node[circle, draw, inner sep=1.3pt] (-M) at (0,.45) {};

      \draw[orange, opacity=.1] (-L) -- (-T);
      \draw[green!60!black] (-L) -- (-M);
      \draw[purple, opacity=.1] (-L) -- (-B);

      \draw[cyan, opacity=.1] (-R) -- (-T);
      \draw[brown] (-R) -- (-M);
      \draw[pink, opacity=.1] (-R) -- (-B);

      \node[fill=white, inner sep=1pt] at (-0.52,0.28) {$d_2$};
      \node[fill=white, inner sep=1pt] at (0.52,0.28) {$f_2$};

      \node at (0,-0.15) {$\vdots$};

      \node[above] at (-M) {$w_2$};
      \node[above=3pt] at (-T) {#1};
      \node[below=3pt] at (-B) {#2};      
      \node[below=12pt] at (-B) {$\scriptstyle \star_{401}\star_{402}$};
    }
  }
}
\tikzset{
  pics/localabsorber/.style n args={5}{
    code={
      \draw (0,0) rectangle (1,2.5);

      \draw (0,2) -- (1,2);
      \draw (0,1.5) -- (1,1.5);
      \draw (0,1) -- (1,1);
      \draw (0,.5) -- (1,.5);

      \node at (0.5,2.25) {#1};
      \node[circle, draw, inner sep=1pt] at (0.5,1.75) {#2};
      \node at (0.5,1.35) {#3};
      \node at (0.5,.75) {#4};
      \node at (0.5,0.25) {#5};
    }
  }
}
\pic (c140) at (-1,-2.75) {localabsorber={$c_1$}{$c_2$}{$\vdots$}{$c_{39}$}{$c_{40}$}};
\node at (.5,-1.5) {$:=$};

\begin{scope}[xshift = 1 cm]
\pic (vw1) at (0,0) {colorgadget={$c_1$}{$d_1$}{$\scriptstyle\star^1_1$}{draw=red,draw opacity=0}{draw=orange}};
\pic (vw2) at (2.5,0) {colorgadget={$c_2$}{$d_2$}{$\scriptstyle\star^1_2$}{draw=yellow!95!black}{draw=green!60!black,draw opacity=0}};
\pic (vw40) at (5.5,0) {colorgadget={$c_{40}$}{$d_{40}$}{$\scriptstyle \star^1_{40}$}{draw=blue,draw opacity=0}{draw=purple}};

\pic (w140) at (9.5,0) {k240={}{}}; 

\draw[dotted] (vw1-right) -- node[lbl,above=2pt,font=\scriptsize] {$\star^4_1$} (vw2-left);
\draw[dotted] (vw2-right) -- node[lbl,above=2pt,font=\scriptsize] {$\star_{18}\ldots\star_{390}$} (vw40-left);
\draw[dotted] (vw40-right) -- node[lbl,above=2pt,font=\scriptsize] {$\star^4_{40}$} (w140-L);
\draw[dotted] (w140-R) -- node[lbl,above=2pt,font=\scriptsize] {$\star^5_1$} (12,0);

\pic (e12) at (0.5,-3) {colorgadget={$f_1$}{$f_2$}{$\scriptstyle \star^2_1$}{draw=cyan}{draw=brown,draw opacity=0}};
\pic (e23) at (3,-3) {colorgadget={$f_2$}{$f_3$}{$\scriptstyle \star^2_2$}{draw=brown,draw opacity=0}{draw=teal}};
\pic (e3940) at (6,-3) {colorgadget={$f_{39}$}{$f_{40}$}{$\scriptstyle \star^2_{39}$}{draw=violet,draw opacity=0}{draw=pink}};

\pic (w12) at (8.5,-3) {vertexgadget={$w_1$}{$w_2$}{$\scriptstyle \star^3_1$}{1}{0}};
\pic (w23) at (10.9,-3) {vertexgadget={$w_2$}{$w_3$}{$\scriptstyle \star^3_2$}{0}{1}};
\pic (w3940) at (14.1,-3) {vertexgadget={$w_{39}$}{$w_{40}$}{$\scriptstyle \star^3_{39}$}{0}{1}};

\draw[dotted] (0,-3) -- (e12-left);
\draw[dotted] (e12-right) -- node[lbl,above=2pt,font=\scriptsize] {$\star^5_2$} (e23-left);
\draw[dotted] (e23-right) -- node[lbl,above=2pt,font=\scriptsize] {$\star_{2063}\ldots\star_{2435}$} (e3940-left);

\draw[dotted] (e3940-right) -- node[lbl,above=2pt,font=\scriptsize] {$\star^6_{1}$} (w12-left);
\draw[dotted] (w12-right) -- node[lbl,above=2pt,font=\scriptsize] {$\star^6_2$} (w23-left);
\draw[dotted] (w23-right) -- node[lbl,above=2pt,font=\scriptsize] {$\star_{2463}\ldots\star_{2815}$} (w3940-left);
\end{scope}
\end{tikzpicture}
\end{adjustbox}
\end{center}
\caption{$(\{c_1,\ldots,c_{40}\},\star)$-absorber, with $c_2$-absorbing path shown. We use the short-hand notation for this gadget given on the left side of the figure in later constructions.}
\label{fig:localcolabs}
\end{figure}

\section{Global absorbers}\label{sec:global_abs}

Recall \Cref{def:absorber} of an $\mathcal{S}$-absorber. 
What we refer to as a `global absorber' is a $\binom{X}{\mu n}$-absorber for a set $X$ of vertices or colors. The `distributive absorption' method, due to Montgomery~\cite{Mont}, constructs a global absorber from a robust local absorption property, as given by \Cref{lem:localvtxabs} and \Cref{lem:localcolorabs}. The method proceeds by constructing a global absorber out of local absorbers by using the following auxiliary robustly-matchable bipartite graph as a template.

\begin{lemma}[Robustly matchable bipartite graph, \cite{ABKPT}, Lemma 4.6]\label{lem:Mont}
For every $0<\beta\leq 1$ and for sufficiently large $m\in\N$ with $\beta m\in\N$, there exists a bipartite graph $H_m$ with parts $X\dot\cup Y$ and $Z$, such that $|X|=(1+\beta)m$, $|Y|=2m$, $|Z|=3m$, $H_m$ has maximum degree at most $40$, and for every $X'\subseteq X$ of size $m$ there exists a perfect matching between $X'\cup Y$ and $Z$ in $H_m$.
\end{lemma}

We separately construct global vertex absorbers and global color absorbers; the proofs are nearly identical.

\begin{lemma}[Vertex Global Absorber]\label{lem:vglobalabsorber}
Let $\frac{1}{n} \ll \ep\ll\beta\leq \mu\ll\nu\ll\alpha \leq 1$.
Let $D$ be an $(n,\alpha,\ep)$-edge-color pseudorandom and $4\nu$-color-gadget pseudorandom digraph.
Let $\gamma n=2825(3\mu n)-3$, and let $\star \in \arcset^{\gamma n}$.
Then for every $V'\sbeq V$ and $\C'\sbeq \C$ with $|V'|,|\C'|\geq (1-\nu)n$ and for every $X\sbeq V$ of size $|X|= (\mu+\beta) n$, there exists an $(\binom{X}{\mu n},\star)$-global absorber with internal vertices and colors in $V'$ and $\C'$.
\end{lemma}

\begin{proof}
By \Cref{lem:Mont}, there exists an auxiliary robustly matchable bipartite graph $H$ with parts $Z'$ and $X'\cup Y'$, where $|Z'|=3\mu n$, $|X'|=(\mu+\beta)n$, $|Y'|=2\mu n$. 

Let $Y\subseteq  V'\setminus X$ be chosen to be of size $|Y|=2\mu n$ arbitrarily. We define bijections $h_X: X'\to X$ and $h_Y: Y'\to Y$ arbitrarily and let $h=h_X\cup h_Y$.

Arbitrarily index the vertices in $Z'$ as $z_1, \dots, z_{3\mu n}$. For each vertex $z_i\in Z'$, observe that $|N(z_i)|\leq 40$. Let $N_i \subseteq X\cup Y$ be of size $|N_i|=40$ such that $h(N(z_i))\subseteq N_i$. 

By \Cref{lem:localvtxabs}, we may successively find internally vertex-disjoint and internally color-disjoint $(N_i,\star^i)$-local vertex absorbers $(A_i,s_i,t_i)$ where $\star^i=(\star_{2825(i-1)+1},\ldots, \star_{2825(i-1)+2822})$ with $V(A_i) \subseteq V'$ and $\C(A_i) \subseteq \C'$ for each $i \in [3\mu n]$, since the number of such absorbers we need is $3\mu n\ll \nu n$. By \Cref{lem:connecting}, we may successively find vertex-disjoint and color-disjoint paths $P_i$ of length three from $t_i$ to $s_{i+1}$ with orientation $(\star_{2825i-2},\star_{2825i-1},\star_{2825i})$ for $i\in [3\mu n-1]$ with internal vertices in $V'\setminus (V(A_1)\cup\ldots V(A_{3\mu n}))$ and colors in $\C'\setminus (\C(A_1)\cup\ldots \C(A_{3\mu n}))$.

We show that $A = A_1 \cup P_1 \cup \cdots \cup P_{3\mu n-1} \cup A_{3\mu n}$ is a $(\binom{X}{\mu n}, \star)$-absorber with starting vertex $s_1$ and ending vertex $t_{3\mu n}$. Let $U\subseteq X$ be of size $|U|=\mu n$ and let $U'=h^{-1}(U)$. By \Cref{lem:Mont} there exists a matching in $H$ between $Z'$ and $U'\cup Y'$. If $z_i$ is matched to $z_i'$ in this matching, we activate $A_i$ along the $h(z_i')$-absorbing path. In total, this yields a rainbow path from $s_1$ to $t_{3\mu n}$ which absorbs exactly the vertices of $U$.
\end{proof}

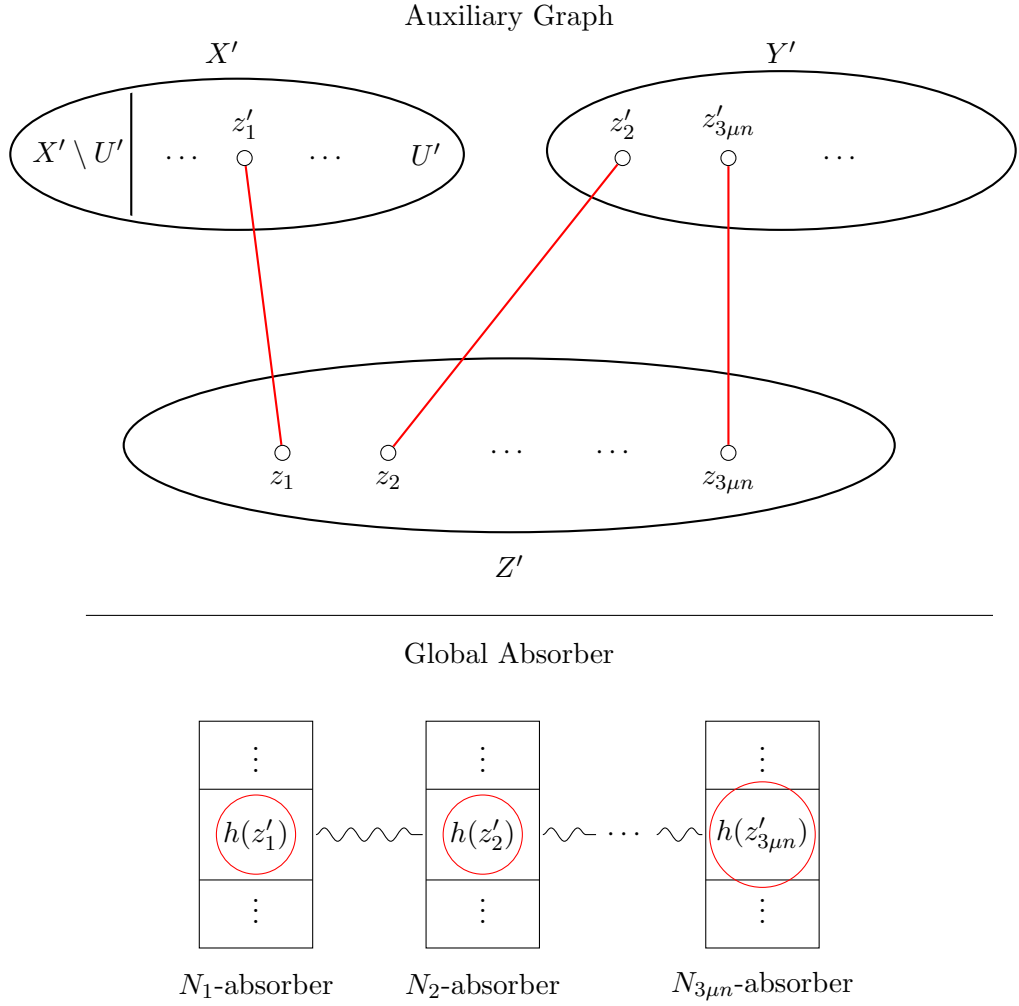
\begin{figure}
    \begin{center}
        \begin{tikzpicture}[
    >=Stealth,
    vertex/.style={circle,draw,inner sep=2pt},
    rededge/.style={red, thick},
    every node/.style={font=\small}
]

\draw[thick] (-4,2.5) ellipse (3 and 1.0);
\node at (-4.2,3.85) {$X'$};
\node at (-1.5,2.5) {$U'$};
\node at (-6.1,2.5) {$X'\setminus U'$};

\draw[thick] (3.2,2.55) ellipse (3.1 and 1.05);
\node at (3.2,3.85) {$Y'$};

\draw[thick] (-0.4,-1.35) ellipse (5.1 and 1.15);
\node at (-0.4,-2.95) {$Z'$};

\begin{scope}
  \clip (-4,2.5) ellipse (2.4 and 1.0);
  \draw[thick] (-5.4,1.65) -- (-5.4,3.35);
\end{scope}

\node[vertex,label=above:$z'_1$] (zp1) at (-3.9,2.45) {};
\node[vertex,label=above:$z'_2$] (zp2) at (1.1,2.45) {};
\node[vertex,label=above:$z'_{3\mu n}$] (zp3) at (2.5,2.45) {};

\node[vertex,label=below:$z_1$] (z1) at (-3.4,-1.45) {};
\node[vertex,label=below:$z_2$] (z2) at (-2,-1.45) {};
\node[vertex,label=below:$z_{3\mu n}$] (z3) at (2.5,-1.45) {};

\draw[rededge] (zp1) -- (z1);
\draw[rededge] (zp2) -- (z2);
\draw[rededge] (zp3) -- (z3);

\node at (-4.7,2.45) {$\cdots$};
\node at (-2.8,2.45) {$\cdots$};

\node at (4,2.45) {$\cdots$};

\node at (-0.4,-1.45) {$\cdots$};
\node at (1,-1.45) {$\cdots$};

\tikzset{
  pics/localabsorber/.style n args={4}{
    code={
      \draw (0,0) rectangle (1.5,3);

      \draw (0,2.1) -- (1.5,2.1);
      \draw (0,.9) -- (1.5,.9);

      \node at (0.75,2.6) {#1};
      \node[draw=red,circle,inner sep=1pt] at (0.75,1.5) {#2};
      \node at (0.75,0.6) {#3};
      \node at (0.75,-0.5) {#4};

      \node (-left) at (0.1,1.5) {};
      \node (-right) at (1.4,1.5) {};
    }
  }
}
\node at (-0.4, -4.1) {Global Absorber};
\node at (-0.4, 4.3) {Auxiliary Graph};

\pic (A1) at (-4.5,-8) {localabsorber = {$\vdots$}{$h(z_1')$}{$\vdots$}{$N_1$-absorber}};
\pic (A2) at (-1.5,-8) {localabsorber = {$\vdots$}{$h(z_2')$}{$\vdots$}{$N_2$-absorber}};
\pic (A3) at (2.2,-8) {localabsorber = {$\vdots$}{$h(z_{3\mu n}')$}{$\vdots$}{$N_{3\mu n}$-absorber}};

\node (dots) at (1.15,-6.5) {$\cdots$};

\draw[decorate, decoration = {snake}] (A1-right) -- (A2-left);
\draw[decorate, decoration = {snake}] (A2-right) --  (dots);
\draw[decorate, decoration = {snake}] (dots) -- (A3-left);

\draw (-6,-3.6) -- (6,-3.6);

\end{tikzpicture}
\end{center}
\caption{$\binom{X}{\mu n}$-absorber, with a particular $U$-absorbing path shown. Local absorbers within global absorber are activated based on matching between $Z'$ and $U'\cup Y'$ in the robustly matchable auxiliary graph $H$.}
\label{fig:globalabsorber}
\end{figure}

\begin{lemma}[Color Global Absorber]\label{lem:cglobalabsorber}
Let $\frac{1}{n} \ll \ep \ll \beta \leq \mu \ll \nu \ll \alpha \leq 1$.
Let $D$ be an $(n,\alpha,\ep)$-edge-color pseudorandom and $4\nu$-color gadget pseudorandom digraph.
Let $\gamma n = 2784(3\mu n)-3$, and let $\star \in \arcset^{\gamma n}$.
Then for every $V'\sbeq V$ and $\C'\sbeq \C$ with $|V'|,|\C'|\geq (1-\nu)n$ and for every $X \subseteq \C$ of size $|X|= (\mu+\beta) n$, there exists an $(\binom{X}{\mu n},\star)$-global absorber with internal vertices and colors in $V'$ and $\C'$.
\end{lemma}

\begin{proof}
By \Cref{lem:Mont} there exists an auxiliary robustly-matchable bipartite graph $H$ with parts $Z'$ and $X'\cup Y'$, where $|Z'|=3\mu n$, $|X'|=(\mu+\beta)n$, $|Y'|=2\mu n$. 

Let $Y\subseteq  \C'\setminus X$ be chosen to be of size $|Y|=2\mu n$ arbitrarily. We define bijections $h_X:X'\to X$ and $h_Y:Y'\to Y$ arbitrarily and let $h=h_X\cup h_Y$.

Arbitrarily index the vertices in $Z'$ as $z_1, \dots, z_{3\mu n}$. For each vertex $z_i\in Z'$, observe that $|N(z_i)|\leq 40$. Let $N_i \subseteq X\cup Y$ be of size $|N_i|=40$ such that $h(N(z_i))\subseteq N_i$. 

By \Cref{lem:localcolorabs}, we may successively find internally vertex-disjoint and internally color-disjoint $(N_i,\star^i)$-local color absorbers $(A_i,s_i,t_i)$ where $\star^i=(\star_{2784(i-1)+1},\ldots, \star_{2784(i-1)+2781})$ with $V(A_i) \subseteq V'$ and $\C(A_i) \subseteq C'$, since the number of such absorbers we need is $3\mu n\ll \nu n$. By \Cref{lem:connecting}, we may successively find vertex-disjoint and color-disjoint paths $P_i$ of length three from $t_i$ to $s_{i+1}$ with orientation $(\star_{2784i-2},\star_{2784i-1},\star_{2784i})$ for $i\in\{1,\ldots, 3\mu n-1\}$ with internal vertices in $V'\setminus (V(A_1)\cup\ldots v(A_{3\mu n}))$ and colors in $\C'\setminus (\C(A_1)\cup\ldots \C(A_{3\mu n}))$.

We show that $A = A_1 \cup P_1 \cup \cdots \cup P_{3\mu n-1} \cup A_{3\mu n}$ is a global absorber with starting vertex $s_1$ and ending vertex $t_{3\mu n}$. Let $U\subseteq X$ be of size $|U|=\mu n$ and let $U'=h^{-1}(U)$. By \Cref{lem:Mont} there exists a matching in $H$ between $Z'$ and $U'\cup Y'$. If $z_i$ is matched to $z_i'$ in this matching, we activate $A_i$ along the $h(z_i')$-absorbing path. In total, this yields a rainbow path from $s_1$ to $t_{3\mu n}$ which absorbs exactly the colors of $U$.
\end{proof}

\section{Proof of \texorpdfstring{\Cref{thm:main}}{Theorem~\ref{thm:main}}}\label{sec:mainproof}

To begin our proof of \Cref{thm:main}, we first set aside a random set $\Vres$ of vertices which we call the `vertex reservoir' and a random set $\Cres$ of colors which we call the `color reservoir.' We then establish global absorbers $A_V$ and $A_\C$ on $\Vres$ and $\Cres$, respectively. On the remainder of the graph, we find an almost spanning rainbow path; what is leftover we incorporate into the path using $\Vres$ and $\Cres$, and then we activate the absorber. The following lemma allows us to incorporate the leftovers using $\Vres$ and $\Cres$; this is a analogue of Lemma 6.1 in~\cite{KLS} to arbitrary orientations, whose proof we give for completeness.

\begin{lemma}[Connecting via reservoirs]\label{lem:res}
Let $1/n \ll \ep \ll \zeta\ll\mu\ll\alpha\leq 1$.
Let $D$ be an $(n,\alpha,\ep)$-edge-color-pseudorandom colored digraph.
Then there exist $\Vres\subseteq V$, $\Cres\subseteq \C$ of size $\mu n$ such that for every distinct $u,v\in V$, for every $c\in\C$, for every $\Vres'\subseteq \Vres$, $\Cres'\subseteq \Cres$ of size at least $(\mu-\zeta)n$, and for every $\star\in\arcset^7$ there exists a rainbow directed path of length seven from $u$ to $v$ with orientation $\star$, with interior vertices in $\Vres'$ and colors in $\Cres'\cup\{c\}$ that contains the color $c$.
\end{lemma}

\begin{proof}
We choose $\Vres\sbeq V$ and $\Cres\sbeq \C$ uniformly at random of size $\mu n$, independently, and show that with high probability, they satisfy the statement of the theorem.

First, we claim that for every $c \in \C$, there exists a matching of size $\mu^2 \alpha n/6$ in $\Vres$ w.h.p. Indeed, for $c \in \C$, by \Cref{def:pseudo} (\ref{pseudo:match}), there is a matching in color $c$ of size $\alpha n/3$; by the Chernoff bound for the hypergeometric distribution, there are at least $\mu^2 \alpha n/6$ edges of this matching in $\Vres$ with probability at least $1-\exp(-\Omega(\mu^2\alpha n))$. The claim follows from a union bound on choice of $c$.

Second, we claim that for every $u \in V$ and $\star \in \arcset$, there exists $W\subseteq \Vres$ of size $|W|\geq \mu^2\alpha n/8$ such that $\{(uw)^{\star}\}_{w\in W}$ forms a rainbow collection in $E(D)$ with high probability. Indeed, for a given $u$ and $\star$, by \Cref{def:pseudo} (\ref{pseudo:colordiversity}), $|\C^{\star}(u)|\geq \alpha n/2$; by the Chernoff bound for the hypergeometric distribution, there are at least $\mu\alpha n/4$ colors $c\in \Cres\cap \C^{\star}(u)$, and therefore at least $\mu^2\alpha n/8$ vertices $v\in N^{\star}_{\Cres}(u)$ such that $\{(uv)^{\star}\}$ forms a rainbow collection, with probability at least $1-\exp(-\Omega(\mu \alpha n))$. The claim follows from a union bound on the choice of $u$ and $\star$.

Now, fix $\Vres$ and $\Cres$ with the above two properties. Let $u,v\in V$ and $c\in\C$, $\Vres'\subseteq \Vres$, and $\Cres'\subseteq \Cres$ each of size at least $(\mu-\zeta)n$, and $\star\in\arcset^7$. By the first property above, $\Vres'$ has a matching in color $c$ of size $(\mu^2\alpha/6 - \zeta )n$, so we can find a matching edge $(xy)^{\star_4}$ with $x,y\in\Vres'$ and $\C((xy)^{\star_4})=c$. 

We apply the second property above to $(u,\star_1)$, $(x,-\star_3)$, $(y,\star_5)$, and $(v,-\star_7)$ to get collections $W_1,W_2,W_3,W_4 \subseteq \Vres$, respectively as in the property. Since $|W_i \cap \Vres'|\geq \mu^2\alpha n/8 - \zeta n \gg 2\ep n$ and $|\Cres'|\geq (\mu -\zeta)n\gg 2\ep n$, we apply \Cref{def:pseudo} (\ref{pseudo:edgecolor}) to get vertex disjoint edges $(w_1w_2)^{\star_2}$ and $(w_3w_4)^{\star_6}$ with distinct colors in $\Cres'$ where the vertices of the collection $w_1,\ldots, w_4$ are pairwise distinct and $w_i\in W_i$. Thus the path $uw_1w_2xyw_3w_4v$ is a rainbow path of length seven with orientation $\star$ and with colors in $\Cres\cup \{c\}$, using the color $c$. 
\end{proof}

\begin{proof}[Proof of \Cref{thm:main}]
Let $1/n\ll1/C\ll\ep\ll \eta\ll \mu\ll \nu\ll \alpha$.
Let $D_0$ be an $n$-vertex digraph with minimum semi-degree $\delta^0(D_0)\geq \alpha n$, and let $D = D_0\cup D(n,C/n)$ be uniformly edge-colored from $[n]$.
By \Cref{lem:ec-pseudorandom} and \Cref{lem:cg-pseudorandom}, we have that $D$ is $(n,\alpha,\ep)$-edge-color pseudorandom and $\nu$-color-gadget pseudorandom with high probability. For the remainder of the proof we assume $D$ satisfies both of these pseudorandom conditions.

We first present the complete proof for the rainbow spanning cycle, and describe afterwards how to prove the existence of cycles of other lengths. Let $\star\in \arcset^n$ be arbitrary. Our goal is to find a rainbow Hamilton cycle with orientation $\star$.

By \Cref{lem:res}, there exists $\Vres\subseteq V$ and $\Cres\subseteq \C$ both of size $(\mu +6\eta)n$ satisfying the hypotheses of the lemma with $\zeta = 6\eta$. Let $\gamma_1=2825(3\mu n)-3$ and $\gamma_2=2784(3\mu n)-3$. By \Cref{lem:vglobalabsorber}, there exists a $(\binom{\Vres}{\mu n},(\star_{1},\ldots,\star_{\gamma_1 n}))$-absorber $(A_{\mathrm{vtx}},s_{\mathrm{vtx}},t_{\mathrm{vtx}})$. By \Cref{lem:cglobalabsorber}, there exists a $(\binom{\Cres}{\mu n},(\star_{\gamma_1n+4},\star_{(\gamma_1+\gamma_2)n+3}))$-absorber $(A_{\mathrm{col}},s_{\mathrm{col}},t_{\mathrm{col}})$ which is vertex-disjoint and color-disjoint from $A_{\mathrm{vtx}}$. We apply \Cref{lem:connecting} to find a rainbow path $P$ from $t_{\mathrm{vtx}}$ to $s_{\mathrm{col}}$ of length three with orientation $(\star_{\gamma_1 n+1},\star_{\gamma_1 n+2},\star_{\gamma_1 n+3})$ and with internal vertices in $V(D)\setminus (V(A_{\mathrm{vtx}})\cup V(A_{\mathrm{col}}))$ and colors in $\C(D)\setminus (\C(A_{\mathrm{vtx}})\cup \C(A_{\mathrm{col}}))$, since $\mu \ll \nu$. Let $A=A_{\mathrm{vtx}}\cup P \cup A_{\mathrm{col}}$. Note that
\[ |V(A)| = \gamma_1 n + 1 + (|\Vres| - \mu n) + 2 + \gamma_2 n + 1 = (\gamma_1 + \gamma_2 + 6\eta)n + 4 \]
and
\[ |\C(A)| = \gamma_1 n + 3 + \gamma_2 n + (|\Cres|-\mu n) = (\gamma_1 + \gamma_2 + 6\eta)n + 3 .\]

By \Cref{lem:longpath}, there exists a rainbow path $Q$ with vertices in $V(D)\setminus V(A)$ and colors in $\C(D)\setminus C(A)$ with orientation $(\star_{(\gamma_1+\gamma_2)n+11},\ldots,\star_{(1-7\eta)n+7})$, so that
\[ |V(Q)| = (1-7\eta)n + 7 - (\gamma_1+\gamma_2)n - 11 + 2 = (1-\gamma_1-\gamma_2-7\eta)n - 2 \]
and
\[ |\C(Q)| = (1-7\eta)n + 7 - (\gamma_1+\gamma_2)n - 11 + 1 = (1-\gamma_1-\gamma_2-7\eta)n - 3 .\]
Let $Q$ have start vertex $x$ and end vertex $y$. Let $V_{\mathrm{rem}}=V\setminus (V(A)\cup V(Q))$ be the set of remaining vertices, which has size
\[ |V_{\mathrm{rem}}| = n - (\gamma_1 + \gamma_2 + 6\eta)n - 4 - (1-\gamma_1-\gamma_2-7\eta)n + 2 = \eta n - 2 ,\]
and let $\C_{\mathrm{rem}}=\C(D)\setminus (\C(A)\cup \C(Q))$ be the set of remaining colors, which has size
\[ |\C_{\mathrm{rem}}| =  n - (\gamma_1 + \gamma_2 + 6\eta)n - 3 - (1-\gamma_1-\gamma_2-7\eta)n + 3 = \eta n .\]

We repeatedly apply \Cref{lem:res} to incorporate $V_{\mathrm{rem}}$ and $\C_{\mathrm{rem}}$ into the cycle. Enumerate the vertices of $V_{\mathrm{rem}}\cup\{s_{\mathrm{vtx}},t_{\mathrm{col}}\}$ as $v_1,\ldots, v_{\eta n}$ with $v_1=t_{\mathrm{col}}$ and $v_{\eta n}=s_{\mathrm{vtx}}$. Furthermore, enumerate the colors of $\C_{\mathrm{col}}$ by $c_1,\ldots, c_{\eta n}$ arbitrarily. By \Cref{lem:res}, there exists a path $P_1$ of length seven from $v_1 = t_{\mathrm{col}}$ to $x$ with orientation $(\star_{(\gamma_1+\gamma_2)n+4}, \dots, \star_{(\gamma_1+\gamma_2)n+10})$ with all $6$ interior vertices in $\Vres$, containing color $c_1$, and all $6$ other colors in $\Cres$. Similarly, by \Cref{lem:res}, there exists a path $P_2$ of length seven from $y$ to $v_2$ with orientation $(\star_{(1-7\eta)n+8}, \dots, \star_{(1-7\eta)n+14})$, vertex- and color-disjoint from $P_1$, with all $6$ interior vertices in $\Vres$, containing color $c_2$, and all $6$ other colors in $\Cres$. For each $2 \leq i \leq \eta n-1$, we similarly apply \Cref{lem:res} to find a path $P_{i+1}$ of length seven from $v_i$ to $v_{i+1}$ with orientation $(\star_{(1-7\eta)n+7i+1},\dots, \star_{(1-7\eta)n+7i+7})$, vertex- and color-disjoint from the previous $P_i$s, with all $6$ interior vertices in $\Vres$, containing color $c_{i+1}$, and all $6$ other colors in $\Cres$. This is possible because the amount of vertices of $\Vres$ and colors of $\Cres$ that we have used by step $i$ is $6i \leq 6 \eta n = \zeta n$.

Let $\Vres'$ be the remaining vertices of $\Vres$ that do not appear on any $P_i$, and similarly let $\Cres'$ be the remaining colors of $\Cres$ that do not appear on any $P_i$, so that $|\Vres'| = |\Cres'| = \mu n$. We activate $A_{\mathrm{vtx}}$ to absorb $\Vres'$ and $A_{\mathrm{col}}$ to absorb $\Cres'$; concatenating all of the relevant paths completes the rainbow cycle with orientation $\star$.

To find a cycle of length $k$ with $(1-3\ep )n \leq k \leq n$, we follow the same proof as the spanning cycle, except that we truncate $Q$ by removing $n-k$ vertices from the end. Since the remainder sets of vertices and colors have $n-k$ more elements than needed, we simply ignore that many vertices and colors when absorbing the remainders with the reservoirs. We finish by activating the global absorbers and concatenating the relevant paths.

To find a cycle of length $k$ with $3 \leq k <(1-3\ep)n$, we do not need to use absorption. Let $\star\in\arcset^k$. We apply \Cref{lem:longpath} to find a rainbow path $P$ of length $k-3$ with orientation $(\star_1,\ldots, \star_{k-3})$. Next, we apply \Cref{lem:connecting} to connect the endpoints of $P$ with a rainbow path $Q$ of length three with orientation $(\star_{k-2},\star_{k-1},\star_k)$ which is internally vertex-disjoint and color-disjoint from $P$. Then $P\cup Q$ is a rainbow cycle of length $k$ with orientation $\star$.

To find a cycle of length $2$, we do not use the pseudorandomness conditions at all, and instead argue directly. Let $\star\in\arcset^2$. Without loss of generality, we take $\star_1=\rightarrow$. It suffices to find a deterministic edge $(vw)^{\rightarrow}\in D_0$ such that $(wv)^{\star_2}\in D(n,C/n)$ and $\C((wv)^{\star_2})\neq \C((vw)^{\rightarrow})$. Let $A_\star$ be the event that this does not hold. Then 
\[\mathbb{P}(A_\star)\leq \left( 1-\frac{C}{n}+\frac{1}{n} \right)^{\alpha n^2}\leq e^{-(C-1)\alpha n}=o(1) ,\]
so for either choice of $\star_2$, we can find a rainbow cycle of orientation $\star$.
\end{proof}

\section{Acknowledgments}

This work was completed for an honors thesis of the second author. We thank Igor Araujo for a helpful discussion and comments which improved the exposition of the paper.


\end{document}